\documentclass{amsart}

\usepackage{mathrsfs}

\usepackage{amsfonts,amssymb,amsmath,amsthm}

\usepackage{color}

\usepackage{enumerate}

\def\NN{{\mathbb N}}

\def\ZZ{{\mathbb Z}}

\def\QQ{{\mathbb Q}}

\def\eps{\varepsilon}

\def\rho{\varrho}

\def\phi{\varphi}

\newcommand{\be}[1]{\begin{equation}\label{#1}}

\newcommand{\ee}{\end{equation}}

\newcommand{\multsum}[2]{\sum_{{\scriptstyle #1}\atop {\scriptstyle #2}}}

\newcommand{\rf}[1]{{\rm (\ref{#1})}}

\newcommand{\Pe}[1]{P_{#1,{\sf e}}}
\newcommand{\Po}[1]{P_{#1,{\sf o}}}
\newcommand{\Se}[1]{S_{#1,{\sf e}}}
\newcommand{\So}[1]{S_{#1,{\sf o}}}
\newcommand{\hatPe}[1]{\widehat{P}_{#1,{\sf e}}}
\newcommand{\hatPo}[1]{\widehat{P}_{#1,{\sf o}}}

\def\cal{\mathscr}

\def\ts{\textstyle}

\def\NN{{\mathbb N}}

\def\QQ{{\mathbb Q}}\def\ZZ{{\mathbb Z}}

\def\bfx{{\bf x}}\def\bfy{{\bf y}}

\def\ups {{\upsilon}}

\newtheorem*{Thm}{Theorem}
\newtheorem*{prop}{Proposition}
\newtheorem{lemma}{Lemma}[section]

\makeatletter

\@addtoreset{equation}{section}
\makeatother

\makeatletter
\def\blfootnote{\xdef\@thefnmark{}\@footnotetext}
\makeatother

\title[Artin's conjecture]
{On Artin's conjecture:\\ linear slices of diagonal hypersurfaces} 

\author{J\"org Br\"udern and Olivier Robert}

\begin{document}

\begin{abstract}
 Artin's conjecture is established for all forms that can be realised as a diagonal form on an hyperplane.
\end{abstract}

\maketitle


\section{Introduction}\label{introduction}
A famous conjecture of Emil Artin asserts that forms of degree $k$ with integer coefficients in $s$ variables have non-trivial zeros in all $p$-adic fields provided only that $s>k^2$. Although the conjecture was disproved a long time ago
(Terjanian \cite{T1,T2})
and desperately fails in some sense, 
it is also not too far from the truth in certain other interpretations: given a degree $k$, there is a number $p_0(k)$ with the property that whenever the prime $p$ exceeds $p_0(k)$ then all forms of degree $k$ with integer coefficients in more than $k^2$ variables have non-trivial zeros in $\mathbb Q_p$
(e.g.\ Ax and Kochen \cite{AK}), while for each prime $p$ there are infinitely many degrees $k$ and  forms of this degree $k$ in more than $\exp \surd k$
variables that have no solution in $\mathbb Q_p$ other than the trivial one (Arkhipov and Karatsuba \cite{kara}, Brownawell \cite{Brow},
Lewis and Montgomery \cite{LM}, Wooley \cite{Tsurvey}).

At the time the conjecture was put forward, ca.\ 1930, it was known to hold when $k=1$ (trivial) and $k=2$ (Meyer \cite{M}). Since then, only the case $k=3$ was settled affirmatively (Demyanov \cite{D}, Lewis \cite{L}, Davenport \cite{D59}). For some other small degrees, the conjecture was confirmed  except for  a concrete list of small primes. For quintic forms, for example, $p$-adic solubility is
guaranteed for all primes $p\ge 11$ (Dumke \cite{Dumk}). Certainly the conjecture was very influential in shaping the subject area, and remains a source of inspiration and inquiry.

One possible line of attack for the original question, or approximations thereof, begins with diagonalisation. Indeed, when interpreted on a suitable $\QQ$-vector space, a form with integer coefficients diagonalises. More precisely, whenever
$f\in\ZZ[x_1,\ldots,x_s]$ is a form of degree $k$, there are a number $r\ge 0$ and integers $a_j$, $b_{ij}$ $(1\le j \le s+r,\, 1\le i\le r)$ with the property that the equation $f(\mathbf x)=0$ is equivalent with the system of equations
\be{0.1} \sum_{j=1}^{s+r} a_j y_j^k = \sum_{j=1}^{s+r} b_{ij} y_j =0
\quad (1\le i\le r).
\ee
In this context, the equation $f(\mathrm x)=0$ is said to be equivalent with the system \rf{0.1} if, for all field extensions $K/\QQ$ the equation $f(\mathrm x)=0$ has solutions $\mathbf x\in K^s\setminus\{{\mathbf 0}\}$ if and only if the system \rf{0.1} admits solutions $\mathbf y\in K^{r+s} \setminus\{{\mathbf 0}\}$.  In Section 2 we present a precise formulation of
 this transformation which, we believe, is part of the folklore, but seems hard to find in the existing literature. In particular, the solubility of  $f(\mathrm x)=0$ over $\QQ_p$ can be discussed by considering \rf{0.1}.

One obvious advantage is that the system \rf{0.1} is amenable to treatment via the combinatorial theory of $p$-groups, as suggested by the work of Br\"udern and Godinho \cite{BG1, BG2}. However, $r$ is often very large, with negative consequences on the technical side of affairs. Also, the form $f$ can be reshaped as \rf{0.1} in many ways, with different values of $r$. Yet there is a smallest such number, say $r(f)$. This invariant measures how far $f$ digresses from a diagonal form where one has $r=0$. We believe that Artin's conjecture holds for all forms  with small $r$. In fact, in an important paper, Davenport and Lewis \cite{DL63} confirm the Artin conjecture for diagonal forms. Here we show that the conjecture is also true in the case $r=1$.
This is a consequence of our main result that we now  announce in a more direct language.

Fix a degree $k$ and a natural number $s$. Suppose that  $a_j$, $b_j$ are integers, and consider the pair of equations
\be{01} a_1 x_1^k + a_2 x_2^k + \ldots + a_s x_s^k = b_1 x_1 + b_2 x_2 + \ldots + b_sx_s =0. \ee 

\begin{Thm} 
Let $s\ge k^2+2$, and let $p$ be a prime. Then there exists a solution $(x_1,\ldots,x_s)\in\QQ_p^s\setminus\{\mathbf 0\}$ of the system \rf{01}. 
\end{Thm}

As an immediate corollary, we note that for a fixed prime $p$, one may allow the coeffients $a_j$, $b_j$ in \rf{01} to be $p$-adic integers, and still conclude as in the theorem. This follows by a routine approximation argument based on the compactness of $\ZZ_p$. 

If Artin's conjecture is known for a particular value of $k$, then that case of the theorem follows by substituting the linear equation in \rf{01} into the other equation. Thus, for $k\le 3$ our result is trivial, but for $k=4$
and many other even degrees Artin's conjecture fails (\cite{kara,T1}). Further,
the observant reader will have already noticed that the case where $k$ is odd  is merely a special case of the main result  in Knapp \cite{K}. The cases where $k\ge 4$ is even are all new. 

The condition on $s$ in our theorem cannot be relaxed, at least when $k+1$ is a prime $p$. 
In fact, in this case, the only $p$-adic solution of the equation
\be{04} 
\sum_{j=0}^{k-1} p^j \sum_{l=1}^k x_{jk+l}^k =0 
\ee
is $x_\nu=0$ $(1\le \nu\le k^2)$ (see \cite{DL63}, p.\ 454), and hence the pair of equations in $k^2+1$ variables 
given by \rf{04} and $x_{k^2}+x_{k^2+1}=0$  has no non-trivial $p$-adic solution.

There is a large body of work on a generalisation of Artin's conjecture to systems of diagonal forms. These take the shape
\be{02} \sum_{j=1}^S a_{ij} x_j^{k_i} = 0 \quad (1\le i \le R) \ee
in which $S$, $R$ and $k_i$ are natural numbers, and $a_{ij}$ are integers. The conjecture asserts that in each $p$-adic field the equations \rf{02} have a non-trivial solution  provided that
\be{03} S>k_1^2 + k_2^2 + \ldots + k_R^2. \ee
Note that our theorem is the case $R=2$, $k_2=1$, and that the system \rf{0.1} is the case $R=r+1$, $k_j=1$ for $j\ge 2$. It is therefore not without interest to compare our result to others concerning the system
\rf{02}. 
Davenport and Lewis \cite{DL69} considered the important special case where the $k_i$ are all equal and were able to  prove the conjecture when $R=2$ and $k_1=k_2$ is odd  \cite{DL2}. When $k_1=k_2$ is even, Br\"udern and Godinho \cite{BG2} confirmed the conjecture for many $k$ (see also Kr\"anzlein \cite{kranz}).
Knapp \cite{K} then showed that the conjecture also holds when $R=2$ and $k_1\neq k_2$ are both odd, while Wooley \cite{Wool} showed that when $R=2$, $k_1=2$ and $k_2=3$ then $S\ge 11$ suffices to ensure $p$-adic non-trivial solutions. For larger $R$  little is known (see \cite{BG1,DL69,LPW}). It is rather remarkable that Wooley \cite{W3} very recently found examples with $R=2$ where the conjecture fails, though such failures have been familiar for large $R$ (see \cite{LM}). From the perspective taken here, our theorem adds to the small stock of examples where a conjecture of Artin's type has been verified for a class of forms of even degree.     

Our proof of the theorem is largely combinatorial. 
In Section 3, we apply a simple contraction argument that will eliminate the linear equation. In this way
we will obtain the theorem already for almost all $k$. Only those values of $k$ that are small powers of $2$, or that are of the form 
 $k=p-1$, $k=p(p-1)$ with $p$ an odd prime  will deny treatment by this first approach.

In subsequent sections we consider the cases $k=p-1$ and $k=p(p-1)$. We begin with reducing the original problem to one on congruences, and to realise this, we establish our own variant of Hensel's lemma in Section 4. In many instances later in the argument,  congruences will be solved by  implicit and explicit uses of the Cauchy-Davenport theorem. The relevant combinatorial tools for this strategy are provided in Section 5. We develop an elementary inverse theory to make economies on the number of variables save in exceptional cases that can be explicitly described. In Section 6 we introduce a natural equivalence relation on the set of equations \rf{01}. Here we are
motivated by the $p$-normalisation of Davenport and Lewis \cite{DL66}, but our approach is different. In a sense it is only the non-linear equation in \rf{01} that is normalised.   

In Sections 7 to 9, we handle the case $k=p-1$. From earlier work on related questions, one would foresee a reduction to a congruence modulo $p$ for which a non-singular solution is then required. The work in Section 7 shows that this approach is only partially successful. There remains a case where
all the solutions of the ambient congruence modulo $p$ are singular. Fortunately, aided by the inverse theory from Section 5,  the systems where this happens may be classified; these are the {\em critical systems} introduced at the end of Section 7. For the critical systems a direct application of a Hensel type lift is not possible. We bypass this difficulty by solving a congruence to a potentially very large power of $p$, and in doing so we invoke aid from variables in the given system where the coefficients are divisible by $p$. These are features in our argument that are absent from earlier work. For more details we refer to Sections 8 and 9. 

In Sections 10 to 12 the case $k=p(p-1)$ is considered. Apart from minor complications in detail, the treatment in this case is along more familiar lines, and in particular, we will always be able to reduce the problem to one on congruences modulo $p^2$ that admit non-singular solutions.

We are then left with the case where $k$ is a power of the prime $p=2$, discussed in Sections 13 to 17. This takes us into  a third stream of ideas. Our work in Section 4 forces us to solve congruences modulo high powers of 2. Our strategy is to lift solutions, modulo $2^l$, to solutions modulo $2^{l+1}$ through the method of {\em contractions}, as introduced by Davenport and Lewis \cite{DL63,DL2}. The details are rather subtle, and the development of their ideas that is required here is best described {\em en cours}. There is a curious feature concerning the case $k=4$ where our main argument collapses. It so happens that for certain normalised forms of degree 4 the routine reduction to congruences leads one into a dead end. For an example where and why this happens, see \rf{P5-2}. To salvage the situation, we turn to an equivalent system that is rather far from normalised but readily seen to admit $2$-adic solutions. Perhaps this is a first glimpse of certain weaknesses in the traditional $p$-normalisation method.      
   
It would be interesting to explore the limitations of the methods presented in this communication. One question is whether our approach yields when $r>1$, and to what extent. Further,
we propose to compute the invariant $r(f)$ for the forms $f$ of degree $4$ that Terjanian 
\cite{T1,T2} used to rebut Artin's conjecture.

\section{A diagonalisation method}
In this section we briefly substantiate a remark made in the introduction, and show
that a form with rational coefficients can always be realised as a diagonal form  on a 
suitable linear subpace of $\QQ^t$, when $t$ is sufficiently large. This is only a special case
of the following result.

\begin{prop} Let $F$ be a field of characteristic $0$, and let $k,s\in\NN$. Suppose that $g\in F[X_1,\ldots,X_s]$
is a form of degree $k$. Then, there exist a number $r$ with 
$$0\le r\le \frac{s(s+1)\ldots (s+k-1)}{k!} , $$ linear forms $L_j\in F[Y_1,\ldots,Y_{r+s}]$
$(1\le j\le r)$ and $c_j\in F$  $(1\le j\le r+s)$ with the property that in any field extension $E/F$ the equation 
$$ g(x_1,\ldots ,x_s)=0$$
has a solution $\mathbf x \in {E}^s\setminus \{\mathbf 0\}$ if and only if the system of equations
$$ \sum_{j=1}^{r+s} c_j y_j^k = 0,\quad L_j(\mathbf y) =0 \quad (1\le j\le r)$$
has a solution $\mathbf y \in {E}^{r+s} \setminus \{\mathbf 0\}$.
\end{prop}

In the sequel, we shall suppose that $F$ and $k$ are as in the hypotheses in the Proposition.

For a proof of the proposition, let $R=s(s+1)\ldots (s+k-1)/(k!)$. Then there are
linear forms $\Lambda_j(X_1,\ldots,X_s)$
with coefficients in $F$, and $\alpha_j\in F$ such that
\be{AA3}   g(X_1,\ldots,X_s) = \sum_{j=1}^R \alpha_j\Lambda_j(X_1,\ldots,X_s)^k. \ee
This is shown in Ellison \cite{Ell}, pp.\ 665--666, over the complex numbers, but the argument works over
fields of characteristic zero.

Now suppose that $(x_1,\ldots,x_s)\in E^s$ is a solution of $g(\mathbf x)=0$ with $\mathbf x\neq 0$. We put
$y_{j}= \Lambda_j(x_{1},\ldots,x_{s})$. Then, the $x_i$ and $y_{j}$ solve the system of equations
\be{AA4} \sum_{j=1}^{R}  \alpha_j y_{j}^k = 0,
\quad y_{j} - \Lambda_j(x_{1},\ldots,x_{s})=0 \quad (1\le j\le R).\ee
Conversely, suppose that
a non-trivial solution in $E$ of \rf{AA4} is given. If this solution would have $x_i=0$ for all $1\le i\le s$, 
then a consideration of the linear subsystem 
shows that also all $y_{j}$ would be $0$ which is not the case. Hence, some
 of  the variables $x_j$ must be non-zero, and from \rf{AA3} we see that $g(\mathbf x)=0$.  This completes the proof.

\section{Contractions}\label{section-contract}

Throughout the paper, we now suppose that $p$ is a prime and that $k\ge 4$. 
We may do so because for $k=1,2,3$ the Artin conjecture is known to hold; recall the comments in Section \ref{introduction}.
In this section, we apply a simple contraction argument to the pair of equations \rf{01}. In short,
in this equation, we force that $b_{2l-1}x_{2l-1}+b_{2l}x_{2l}=0$,
parametrize the solutions of this linear relation,   and substitute into the degree $k$ equation. We are then left with a single equation of degree $k$ in $[s/2]$ variables. In many cases, this argument is of strength sufficient to conclude that \rf{01} has non-trivial $p$-adic solutions.
\smallskip

 Let  $\Gamma^{*}(k,p)$ denote the smallest natural number $t$ with the property that whenever $c_1,\dots, c_t\in \ZZ$, then the equation
\be{11}
c_1x_1^k+c_2x_2^k+\dots+c_tx_t^k=0
\ee
has a non-trivial solution $\bfx\in \QQ_p^t$. The following lemma makes the contraction argument precise.

\begin{lemma}\label{contract1}
Suppose that $s\ge 2\Gamma^{*}(k,p)$. Then the system \rf{01} has a non-trivial solution in $\QQ_p$.
\end{lemma}
\begin{proof} Within this proof, let $\Gamma=\Gamma^{*}(k,p)$. For $1\le l\le \Gamma$, define the integers $u_{2l-1},u_{2l}$ by 
$$
u_{2l-1}=b_{2l},\qquad u_{2l}=-b_{2l-1}
$$
except when $b_{2l-1}=b_{2l}=0$ in which case we take $u_{2l-1}=u_{2l}=1$. Then in all cases, one at least of $u_{2l-1},u_{2l}$ is non-zero. With $y_l\in \QQ_p$ still to be determined, we now choose
\be{12}
x_{2l-1}=u_{2l-1}y_l,\qquad x_{2l}=u_{2l}y_l\qquad (1\le l\le \Gamma)
\ee
and then put $x_j=0$ for $2\Gamma<j\le s$. Then
$$
\sum_{j=1}^sb_jx_j=\sum_{l=1}^{\Gamma}y_l(b_{2l-1}u_{2l-1}+b_{2l}u_{2l})=0
$$
and
$$
\sum_{j=1}^sa_jx_j^k=\sum_{l=1}^{\Gamma}c_ly_l^k
$$
in which $c_l=a_{2l-1}u_{2l-1}^{k}+a_{2l}u_{2l}^k\in \ZZ$. We choose a solution $\bfy\in \QQ_p^{\Gamma}\smallsetminus \{\boldsymbol{0}\}$ of $c_1y_1^k+\dots +c_{\Gamma}y_{\Gamma}^k=0$. With this choice of $\bfy$, the numbers $\bfx\in \QQ_p^s$ defined in \rf{12} are a non-trivial solution of \rf{01}.
\end{proof}

With Lemma \ref{contract1} in hand, we wish to determine conditions on $p$ that ensure
\be{13}
\Gamma^{*}(k,p)\le \tfrac{1}{2}k^2+1,
\ee
because in such circumstances the conclusion of the Theorem is implied at once.

The function $\Gamma^{*}(k,p)$ has been studied in detail by Dodson \cite{Do}. We proceed by discussing the consequences of his work for $2$-adic solubility.

\begin{lemma}\label{contract2}
Let $k\ge 5$, but not one of the numbers $8,16,32$. If $s\ge k^2+2$, then  the equations \rf{01} have a non-trivial $2$-adic solution.
\end{lemma}
\begin{proof}
First suppose that $k$ is odd. Then by Dodson \cite[Lemma 4.2.2]{Do},  we have $\Gamma^{*}(k,2)=k+1$. Hence \rf{13} holds.

Next we suppose that $k$ is even and write $k=2^{\tau}k_0$ with $\tau\ge 1$ and $k_0$ odd. Then by Dodson \cite{Do}, Lemma 4.6.1, one has
\be{14}
\Gamma^*(k,2)\le \Big[\frac{k(2^{\tau+2}-1)}{\tau+2}\Big]+1.
\ee

If $k_0=1$ and $\tau\ge 6$, one has
$$
\frac{k(2^{\tau+2}-1)}{\tau+2}= \frac{4k^2-k}{\tau+2}<\tfrac{1}{2}k^2,
$$
so that \rf{14} implies \rf{13}. If $k_0\ge 3$ and $\tau\ge 1$, one finds that
$$
\frac{k(2^{\tau+2}-1)}{\tau+2}\le \frac{4k2^{\tau}}{\tau+2}= \frac{4}{(\tau+2)k_0}k^2\le \tfrac{4}{9}k^2.
$$
Again, via \rf{14}, this confirms \rf{13}. We have now shown that for all $k$ covered by the hypotheses in Lemma \ref{contract2}, the inequality \rf{13} holds. The conclusion of  Lemma \ref{contract2} now follows from Lemma \ref{contract1}.
\end{proof}

A similar argument applies when $p$ is odd. In this context, put $d=(k,p-1)$ and  write
\be{15}
k=p^{\tau}dk_0
\ee
with $p\nmid k_0$. Dodson \cite[p. 165]{Do} denotes by $\gamma^*(k,p^l)$ the smallest positive  integer $t$ with the property that whenever $c_1,\dots, c_t$ are integers coprime to $p$ then the congruence
$$
c_1x_1^k+c_2x_2^k+\dots+c_tx_t^k\equiv 0\bmod p^l
$$
has a solution with at least one of $x_1,\dots, x_t$ coprime to $p$. Further progress will depend on the inequality
\be{16}
\Gamma^*(k,p)\le k\big(\gamma^*(k,p^{\tau+1})-1\big)+1
\ee
that is part of \cite[Lemma 4.2.1]{Do}.

We note that Dodson, \cite [Lemma 2.3.2]{Do}    obtained the estimate
\be{17}
\gamma^*(\delta,p)\le \big[\tfrac{1}{2}(\delta+4)\big]
\ee
whenever $\delta\mid p-1$, $\delta<\tfrac{1}{2}(p-1)$ and $p\ge 5$, irrespective of the parity of $\delta$.
\begin{lemma}\label{C3}
Suppose that the even natural number $\delta$ satisfies the relations $\delta\mid p-1$ and $\delta<\tfrac{1}{2}(p-1)$. Then $\gamma^*(\delta,p)\le \tfrac{1}{2}\delta+1$.
\end{lemma}
\begin{proof} The hypotheses imply that $p\ge 7$. Now suppose that $2t>\delta$, and choose $c_1,\dots, c_t$ coprime to $p$. Then, by a familiar result of Chowla, Mann and Straus \cite{CMS} (or \cite[Theorem 2.8]{N}), the set
$$
R_0=\Big\{\sum_{j=1}^{\delta/2}c_jx_j^{\delta}\colon \quad x_j\in \mathbb{F}_p\quad (1\le j\le \delta/2)\Big\}
$$
contains at least $(\delta-1)\frac{p-1}{\delta}+1$ elements. Put $R=R_0\smallsetminus\{0\}$. By Lemma 2.11 of Nathanson \cite{N}, the set $S=\{c_tx^{\delta}\colon x\in \mathbb{F}_p\}$ is not an arithmetic progression in $\mathbb{F}_p$, and the theory of power residues shows that $\#S=\frac{p-1}{\delta}+1$.
Hence, for computing the size of the sumset $R+S$,  Vosper's theorem \cite[Theorem 2.7]{N} combines with the Cauchy-Davenport theorem \cite[Theorem 2.1]{N}, and we find that
   $\#(R+S)\ge \min(p, \#R+\#S)$. The lower bounds for the sizes of $S$ and $R$ show that
$$
 \#R+\#S\ge (\delta-1)\frac{p-1}{\delta}+\frac{p-1}{\delta}+1=p.
$$
In particular, $0\in R+S$, and hence, there is a solution of $c_1x_1^{\delta}+\dots c_tx_t^{\delta}\equiv 0\bmod p$ with at least one of $x_1,\dots,x_{\delta/2}$ not divisible by $p$.
\end{proof}
\begin{lemma}\label{C4}
Let $k\ge 4$ be even, and let $p$ be an odd prime with $p\nmid k$ and $p-1\neq k$. Then, whenever $s\ge k^2+2$, the equations \rf{01} have a non-trivial solution in $\QQ_p$.
\end{lemma}
\begin{proof}
In \rf{15}, we have $\tau=0$. Note that $\gamma^{*}(k,p)=\gamma^{*}(d,p)$ (see \cite[(2.1.2)]{D}). Thus, we may use  upper bounds for $\gamma^{*}(d,p)$ in \rf{16} to verify \rf{13}.

We divide into cases. First suppose that $d$ is even and that $d<\tfrac{1}{2}(p-1)$. Then, since $d\mid (p-1)$, we may apply Lemma \ref{C3} to conclude that $\gamma^*(k,p)\le \tfrac{1}{2}d+1$. However, $d\mid k$, and hence $\gamma^*(k,p)\le \tfrac{1}{2}k+1$. Now \rf{16} implies \rf{13}.

Next, suppose $d$ is odd and $d<\tfrac{1}{2}(p-1)$. Then, by \rf{17}, we have $\gamma^*(k,p)\le \tfrac{1}{2}(d+3)$. But the odd number $d$ divides the even number $k$, whence $d\le \tfrac{1}{2}k$, and \rf{16} produces
$$
\Gamma^*(k,p)\le k\big(\tfrac{1}{4}k+\tfrac{1}{2}\big)+1\le \tfrac{1}{2}k^2+1
$$
as desired.

We now consider $d=p-1$. By \rf{15} and the hypothesis that $p-1\neq k$, we have $k=(p-1)k_0$ with $k_0\ge 2$. Further, by \cite[(2.3.2)]{D},  one has $\gamma^*(p-1,p) =p$. By \rf{16}, this yields
$$
\Gamma^*(k,p)\le k(p-1)+1=k_0^{-1}k^2+1\le \ts\frac{1}{2}k^2+1.
$$
This again confirms \rf{13}.

This leaves the case $d=\tfrac{1}{2}(p-1)$ for consideration. In this situation, we deduce from $d\mid k$ that $p\le 2k+1$. Moreover, \cite[Lemma 2.2.1]{Do} supplies the bound
\be{18}
\gamma^*\big(\tfrac{1}{2}(p-1),p\big)= \big[\frac{\log p}{\log 2}\big]+1.
\ee
But then, since
$$
\frac{\log p}{\log 2}\le \frac{\log (2k+1)}{\log 2}<\tfrac{1}{2}k+1
$$
holds for all $k\ge 6$, we conclude from \rf{18} that $\gamma^*\big(\tfrac{1}{2}(p-1),p\big)\le \tfrac{1}{2}k+1$ for these $k$, and then from \rf{16} that \rf{13} holds. When $k=4$, the condition $d=\tfrac{1}{2}(p-1)$ holds for no prime $p$.
\end{proof}

\begin{lemma}\label{C5}
Let $k\ge 6$ be even, and let $p$ be an odd prime with $p\mid k$. If $k\neq p(p-1)$ and $s\ge k^2+2$, then the equations \rf{01} have a non-trivial $p$-adic solution.
\end{lemma}

\begin{proof} We again consider cases, depending on the size of $d$. If $d=p-1$, then \cite[Lemma 4.6.1]{Do}  shows that
\be{19}
\Gamma^*(k,p)\le \Big[\frac{k(p^{\tau+1}-1)}{\tau+1}\Big]+1.
\ee
But now $k=p^{\tau}(p-1)k_0$ with $\tau\ge 1$. For $\tau\ge 2$ we note that
\begin{align*}
\frac{p^{\tau+1}-1}{\tau+1}&\le \frac{1}{3}(p^{\tau+1}-1)= \frac{1}{3}\big(p^{\tau}(p-1)+p^{\tau}-1\big)\\
&\le \frac{1}{3}(k+p^{\tau})\le \frac{1}{3}k\Big(1+\frac{1}{p-1}\Big)\le \frac{1}{2}k.
\end{align*} 
Hence $\Gamma^*(k,p)\le \tfrac{1}{2}k^2+1$. This gives \rf{13}. For $\tau=1$ the hypothesis in Lemma \ref{C5} implies $k_0\ge 2$, and then
\begin{align*}
\frac{p^{\tau+1}-1}{\tau+1}&=\frac{1}{2}(p^2-1)=\frac{1}{2}\big(p(p-1)+p-1\big)=\frac{1}{2}\Big(\frac{k}{k_0}+p-1\Big)\\
&=\frac{1}{2}k\Big(\frac{1}{k_0}+\frac{1}{pk_0}\Big)\le \frac{1}{2}\cdot\frac{4}{3}\cdot\frac{k}{k_0}\le \frac{1}{3}k,
\end{align*} 
which again implies \rf{13} via \rf{19}.

It remains to consider the range  $2\le d\le \tfrac{1}{2}(p-1)$. We put $\nu=\gamma^*(d,p)$. By \cite[Lemma 4.3.2]{Do}, we have
\be{110}
\Gamma^*(k,p)\le \Big[\frac{k(\nu^{\tau+1}-1)}{\min(\nu,\tau+1)}\Big]+1.
\ee

First suppose that $d<\tfrac{1}{2}(p-1)$. We begin by showing that in this case  one has
\be{110b}
\nu\le \tfrac{1}{2}d+\tfrac{3}{2}\quad\mbox{ and }\quad\nu\le d.
\ee
In fact, the first of these inequalities is \rf{17} when $d$ is odd, while Lemma \ref{C3} asserts that $\nu\le \frac{1}{2}d+1$ when $d$ is even. In the latter  case, the hypotheses that $d\ge 2$ implies that $\frac{1}{2}d+1\le d$, confirming \rf{110b} for even values of $d$. When $d$ is odd, one has $d\ge 3$, and hence, it follows that $\frac{1}{2}d+\frac{3}{2}\le d$, again confirming \rf{110b}.

From \rf{110b} and the trivial bound $\nu\ge 2$ we now infer that  
$$
\frac{\nu^{\tau+1}-1}{\min(\nu,\tau+1)}\le \frac{1}{2}\nu^{\tau+1}\le \frac{1}{2}d\Big(\frac{1}{2}d+\frac{3}{2}\Big)^{\tau}.
$$
But $d\mid (p-1)$ and $d<\tfrac{1}{2}(p-1)$ so that $d\le \frac{1}{3}(p-1)$. Therefore
$$
 \frac{\nu^{\tau+1}-1}{\min(\nu,\tau+1)}\le\frac{d}{2}\Big(\frac{p-1}{6}+\frac{3}{2}\Big)^{\tau}\le \frac{1}{2}dp^{\tau}\le \frac{1}{2}k.
$$
Now \rf{110} implies \rf{13}.

This leaves the case where $d=\tfrac{1}{2}(p-1)$ and $d\ge 2$. Note that now $p\ge 5$, and then
$$
\frac{\log p}{\log 2}<\frac{1}{2}(p+1)=d+1,
$$
as is easily checked. By \rf{18}, it follows that $\nu\le d+1$. Hence by recalling again that $\nu\ge 2$, 
$d\ge 2$, $p\ge 5$ and $k_0\ge 1$,
we now infer that
\begin{align*}
\frac{\nu^{\tau+1}-1}{\min(\nu,\tau+1)}&\le \frac{1}{2}(\nu^{\tau+1}-1)\le  \frac{1}{2}(d+1)\Big(\frac{p+1}{2}\Big)^{\tau}\\
&= dp^{\tau}\Big(\frac{1}{2}+\frac{1}{2d}\Big)\Big(\frac{1}{2}+\frac{1}{2p}\Big)^{\tau}\le \frac{9}{20}\,k.
\end{align*}
Once again, \rf{110} implies \rf{13}, and the lemma follows.
\end{proof}

We summarise the results obtained so far. For odd $k$, the conclusion in our theorem is contained in Knapp \cite{K}. For even $k$, the theorem also follows from Lemmas \ref{contract2}, \ref{C4} and \ref{C5} except for the following situations:
\be{REST}
p=2,\thinspace  k\in \{4,8,16,32\},\qquad p>2,\thinspace k=p-1\mbox{ or }\, k=  p(p-1).
\ee

\section{Reduction to congruences}

In this section, we reduce the question concerning $p$-adic solubility to suitable congruences. This is achieved via an appropriate version of Hensel's lemma that we formulate as Lemma \ref{OR} below. Throughout this section, we suppose that 
\be{R1}
k=p^{\tau}(p-1)
\ee
holds with some $\tau\in \NN_0$, and that $k\ge 4$. Hence for $p=2$, this implies $\tau\ge 2$. It is important to note that the cases listed in \rf{REST} are all of the form \rf{R1}. We  put
\be{R2}
\gamma=\tau+1\mbox{ except when }p=2\mbox{ where }\gamma=\tau+2.
\ee

\begin{lemma}\label{OR} Let $p$ be a prime, and suppose that $k$ is linked with $\tau$  via \rf{R1}. Let $a_1,a_2,b_1,b_2,A,B$ and $x_1,x_2$ denote integers satisfying
\be{R3}
a_1x_1^k+a_2x_2^k\equiv A\bmod p^{\gamma}\quad \mbox{ and }\quad b_1x_1+b_2x_2=B
\ee
with
\be{R4}
p\nmid b_1a_2x_2^{k-1}-b_2a_1x_1^{k-1}.
\ee
Then there are $y_1,y_2\in \ZZ_p$ with $(y_1,y_2)\neq (0,0)$ and
\be{05}
a_1y_1^k+a_2y_2^k=A\quad\mbox{ and }\quad b_1y_1+b_2y_2=B.
\ee
\end{lemma}
\noindent
In the sequel, we refer to solutions of \rf{R3} that satisfy \rf{R4} as {\em non-singular}.

\begin{proof}
By \rf{R4} the prime $p$ cannot divide $b_1a_2x_2^{k-1}$ and $b_2a_1x_1^{k-1}$ simultaneously. By symmetry in the indices $1$ and $2$, we may therefore suppose that
\be{R5-1}
p\nmid b_1a_2x_2.
\ee
Now multiply the congruence in \rf{R3} with $b_1^k$, and put $z_1=b_1x_1$. Then \rf{R3} transforms into
\be{R6}
a_1z_1^k+a_2b_1^kx_2^k\equiv Ab_1^k\bmod p^{\gamma},\quad z_1+b_2x_2=B,
\ee
and elimination of $z_1$ yields the congruence
\be{R7}
a_1(B-b_2x_2)^k+a_2b_1^kx_2^k\equiv Ab_1^k\bmod p^{\gamma}.
\ee 
Now consider the polynomial $\phi\in \ZZ[t]$ defined by
\be{R8}
\phi(t)=a_1(B-b_2t)^k+a_2b_1^kt^k-Ab_1^k.
\ee
Its formal derivative is 
\be{R9}
\phi'(t)=k\big( a_2b_1^k  t^{k-1}-a_1b_2(B-b_2t)^{k-1}\big).
\ee
By  \rf{R7}, one has $\phi(x_2)\equiv 0\bmod p^{\gamma}$. Furthermore, by \rf{R9}, we infer that
\begin{align*}
\frac{\phi'(x_2)}{k}&= a_2b_1^kx_2^{k-1}-a_1b_2(B-b_2x_2)^{k-1}\\
&=a_2b_1^kx_2^{k-1}-a_1b_2(b_1x_1)^{k-1}=b_1^{k-1}\big(a_2b_1x_2^{k-1}-a_1b_2x_1^{k-1}\big),
\end{align*}
thus showing via \rf{R4} and \rf{R5-1}  that $p^{\tau}\|\phi'(x_2)$.

We now construct integers  $\xi_l$, starting with $\xi_{\gamma}=x_2$, that satisfy the relations
\be{R10}
\phi(\xi_l)\equiv 0\bmod p^{l},\quad \xi_{l+1}\equiv \xi_l\bmod p^{l-\tau}
\ee
for all $l\ge \gamma$. To achieve this, suppose that $\xi_l$ is already determined and put $\xi_{l+1}=\xi_l+p^{l-\tau}h$, with $h\in \ZZ$ at our disposal. Then, by Taylor's theorem,
$$
\phi(\xi_{l+1})=\phi(\xi_l)+\phi'(\xi_l)p^{l-\tau}h+\sum_{j=2}^{k}\frac{\phi^{(j)}(\xi_l)}{j!}p^{j(l-\tau)}h^j.
$$
An inspection of \rf{R9} reveals the $k\mid \phi^{(j)}(\xi_l)$ for all $j\ge 1$, and that $\phi^{(j)}(\xi_l)/j!$ is an integer. Further, taking into account the exact power of $p$ that divides $j!$ it easily follows that 
$p^{l+1} $
divides  $\phi^{(j)}(\xi_l)p^{j(l-\tau)}/j!$
for  all $j\ge 2$ and all $l\ge \gamma$. In particular, we now see that there is an integer $d$ with
\be{R11}
\phi(\xi_{l+1})=p^{l}\Big(\frac{\varphi(\xi_l)}{p^l}+\frac{\phi'(\xi_l)}{p^{\tau}}h\Big)+p^{l+1}d.
\ee
An appropriate choice of $h$ in \rf{R11} gives $\phi(\xi_{l+1})\equiv 0\bmod p^{l+1}$ while the recursive congruence in \rf{R10} arises from the construction.

By \rf{R10}, we also see that the sequence $\xi_l$ converges to a limit $y_2\in \ZZ_p$, and one has
$
\phi(y_2)=0$ and $y_2\equiv x_2\bmod p$,
so that  \rf{R5-1} then gives $y_2\in \ZZ_p^{\times}$. We now define $y_1\in \QQ_p$ by $b_1y_1+b_2y_2=B$. But $p\nmid b_1$ (by \rf{R5-1}), so that $y_1\in \ZZ_p$. By \rf{R8},
$$
0=\phi(y_2)= a_1(B-b_2y_2)^k+a_2b_1^ky_2^k-Ab_1^k
=b_1^k\big(a_1y_1^k+a_2y_2^k-A\big).
$$
This completes the proof of the lemma.
\end{proof}

Let $a_1,\dots,a_s,b_1,\dots, b_s$ be integers, and consider the forms 
\be{3-2}
A(x_1,\dots,x_s)=\sum_{j=1}^sa_jx_j^k,\qquad B(x_1,\dots, x_s)=\sum_{j=1}^sb_jx_j.
\ee

\begin{lemma}\label{lemmaH} Let $s\ge 2$, and suppose that ${\bf x}\in \ZZ^s$ satisfies the congruences
\be{R12}
A({\bf x})\equiv 0\bmod p^{\gamma},\qquad B({\bf x})\equiv 0\bmod p
\ee
and \rf{R4}. Then there are $y_1,y_2\in \ZZ_p$ with $(y_1,y_2)\neq (0,0)$ and
$$
A(y_1,y_2,x_3,\dots, x_s)=B(y_1,y_2,x_3,\dots, x_s)=0.
$$ 
\end{lemma}
\begin{proof}
Put 
$$
A=-\sum_{j=3}^sa_jx_j^k,\qquad B=-\sum_{j=3}^sb_jx_j.
$$
Then \rf{R12} becomes 
\be{R13}
a_1x_1^k+a_2x_2^k\equiv A\bmod p^{\gamma},\qquad b_1x_1+b_2x_2\equiv B\bmod p,
\ee
while  \rf{R4} implies that $p$ cannot divide both $a_1x_1b_2$ and $a_2x_2b_1$. On exchanging the roles of the indices $1$ and $2$ if necessary, we may assume henceforth that $p\nmid b_1a_2x_2$.

Let $q=(b_1;b_2)$. Then $p\nmid q$, and the substitution $z_j=qx_j$ takes \rf{R13} to 
\be{R14}
a_1z_1^k+a_2z_2^k\equiv Aq^k \bmod p^{\gamma},\qquad b'_1z_1+b'_2z_2\equiv B\bmod p
\ee
in which $b'_j=b_j/q$. By \rf{R14} there is an integer $c$ with $b'_1z_1+b'_2z_2=B-pc$. Since $(b'_1;b'_2)=1$, there are $u_1,u_2\in \ZZ$ with $b'_1u_1+b'_2u_2=c$. We take $w_j=z_j+pu_j$. Then 
\be{R15}
b'_1w_1+b'_2w_2=B
\ee
while
$$
w_j^k=(z_j+pu_j)^k=z_j^k+kz_j^{k-1}pu_j+\tfrac{1}{2}k(k-1)z_j^{k-2}p^2u_j^2+\dots
$$
For odd $p$, we see that
$
w_j^k\equiv z_j^k\bmod p^{\tau+1},
$
and recalling that $\gamma=\tau+1$, we get
\be{R16}
a_1w_1^k+a_2w_2^k\equiv Aq^k\bmod p^{\gamma}.
\ee
In the case where $p=2$ one has $\gamma=\tau+2$ and $k=2^{\tau}$. But then,  binomial expansion shows that there is some $v\in \ZZ$ with
$$
w_j^k=z_j^k+2^{\tau+1}z_j^{k-1}u_j+2^{\tau+1}(k-1)z_j^{k-2}u_j^2+2^{\tau+2}v.
$$
But $k-1$ is odd, and so, $2\mid z_j^{k-1}u_j+(k-1)z_j^{k-2}u_j^2$, and $w_j^k\equiv z_j^k\bmod 2^{\tau+2}$. Again, we arrive at \rf{R16}. We have now verified \rf{R15} and \rf{R16} in all cases.

We wish to apply Lemma \ref{OR}, and therefore consider 
\begin{align*}
b'_1a_2w_2^{k-1}-b'_2a_1w_1^{k-1}&\equiv b'_1a_2z_2^{k-1}-b'_2a_1z_1^{k-1}\\
&\equiv q^{k-1}(b'_1a_2x_2^{k-1}-b'_2a_1x_1^{k-1})\\
&\equiv  q^{k-2}(b_1a_2x_2^{k-1}-b_2a_1x_1^{k-1})\bmod p.
\end{align*}
By \rf{R4}, we conclude that $p\nmid b'_1a_2w_2^{k-1}-b'_2a_1x_1^{k-1}$ as required in Lemma \ref{OR}. This now supplies $y'_1,y'_2\in \ZZ_p$, not both zero, with
$$
a_1{y}_1^{\prime k}+a_2{y}_2^{\prime k}=Aq^k , \quad b'_1y'_1+b'_2y'_2= B.
$$
But $q\in \ZZ_p^{\times}$, so that  the numbers $y_j$ defined by $y'_j=qy_j$ are still in $\ZZ_p$ and satisfy
$a_1y_1^k+a_2y_2^k=A$ and  $b_1y_1+b_2y_2= B$, as required.
\end{proof}

\section{Auxiliaries}

For convenience of the reader, we state here Chowla's extension of  the Cauchy-Davenport theorem, see \cite[Theorem 2.1]{N}.
\begin{lemma}\label{lemCDorg}
Let $q\ge 1$ be an integer. Let $\cal{A},\cal{B}\subset \ZZ/q\ZZ$, and suppose  that $0\in \cal{B}$ and $\cal{B}\setminus \{0\}\subset (\ZZ/q\ZZ)^{\times}$. Let $\cal{A}+\cal{B}$ denote the set of all sums  $a+b$ with $a\in \cal{A}$ and $b\in \cal{B}$. Then
$
\#(\cal{A}+\cal{B})\ge \min (\#\cal{A}+\#\cal{B}-1,q).
$
\end{lemma}

The following simple consequence is frequently used below.

  \begin{lemma}\label{lemCD}  Let $q\ge 2$  be an integer. Let $s\ge q$, and let $c_1,\ldots, c_s
    \in (\ZZ/q\ZZ)^{\times}$. Then, there is a subset $J$ of $\{1,2,\ldots,s\}$ with $1\in J$ and
    $$ \sum_{j\in J}c_j \equiv 0\bmod q.$$
\end{lemma}  

\begin{proof}  Let $\cal{A}_j=\{0,c_j\}$ for $2\le j\le q$. Then, recursive application of
Lemma \ref{lemCDorg} implies that $\cal{A}_2+\cal{A}_3+\dots +\cal{A}_{q}=\ZZ/q\ZZ$.
 Hence there exists $(\eps_j)_{2\le j\le q}$ with $\eps_j=0$ or $1$ such that $\sum_{j=2}^{q}c_j\eps_j=-c_1$. We take $J$ consisting of $1$ and all $j$ with $\eps_j=1$ to confirm the conclusion of the lemma.
\end{proof}

\begin{lemma}\label{L1} Let $p\ge 3$ and $k=p^{\tau}(p-1)$ with $\tau\ge 0$. Let $a_1,\dots,a_{p}\in \mathbb{F}_p^{\times}$. Then there is a solution of $a_1x_1^k+\dots +a_{p}x_{p}^k=0$  in $\mathbb{F}_p$  with $x_1=1$.
\end{lemma}

\begin{proof} Apply Lemma \ref{lemCD} with $q=p$ and take $x_j=1$ for $j\in J$ and $x_j = 0$ otherwise. 
\end{proof}

\begin{lemma} \label{L2} Let $k$ be as in Lemma \ref{L1}.  Suppose that $a_1,\dots,a_{p-1}\in \mathbb{F}_p^{\times}$, and that $a_1x_1^k+\dots +a_{p-1}x_{p-1}^k=0$ has no non-trivial solution. Then the $a_j$ are all equal.
\end{lemma}

\begin{proof} Suppose that the $a_j$ are not all equal, and that $a_1\neq a_2$, say. Then $a_1+a_2\neq 0$ (otherwise, by choosing $x_1=x_2=1$ and the other $x_j$ zero, we would have a non-trivial solution). Hence, by setting  $\cal{A}_j=\{0,a_j\}$ for $1\le j\le p-1$, we have $\#(\cal{A}_1+\cal{A}_2)\ge 4$, and  repeated use of Lemma \ref{lemCDorg}  yields $\cal{A}_1+\cal{A}_2+\dots +\cal{A}_{p-2}=\mathbb{F}_p$. In particular, there exist
$\eps_j\in \{0,1\}$ such that $\sum_{j=1}^{p-2}a_j\eps_j+a_{p-1}=0$. We have constructed a non-trivial solution, which is a contradiction. Thus, the $a_j$ are all equal.
\end{proof}

\begin{lemma} \label{L3} Let $p\ge 3$. Let $a_1,a_2,a_3\in \mathbb{F}_p^{\times}$. Then, at least one of the sums  $a_1+a_2$, $a_1+a_3$, $a_2+a_3$ is non-zero. Moreover, two of these sums are non-zero except when, up to permutation, we have $a_1=a_2=-a_3$.
\end{lemma}
\begin{proof} Trivial.
\end{proof}

\begin{lemma} \label{L4}Let $p\ge 3$. Let $a_1,\dots,a_{p},c \in \mathbb{F}_p^{\times}$, and let $b_1,\dots,b_{p}\in \mathbb{F}_p$. Then there is a non-singular solution in $\mathbb{F}_p$ of the pair of equations
$$
\sum_{j=1}^{p}a_jx_j^{p-1}= cy+\sum_{j=1}^{p}b_jx_j=0.
$$
\end{lemma}

\begin{proof} By Lemma \ref{L1}, there exists a non-trivial solution to $\sum_{j=1}^{p}a_jx_j^{p-1}=0$. Since $c\neq 0$, there exists $y$ such that $cy=-\sum_{j=1}^{p}b_jx_j$. The solution $(x_1,\dots,x_{p},y)$ of the system is non-singular: indeed, since $x_1=1$, the Jacobian for the variables $x_1$ and $y$ is non-zero. 
\end{proof}

\begin{lemma}\label{L5}  Let $p\ge 3$.  Let $a_1,\dots,a_{p-1}\in \mathbb{F}_p^{\times}$, and let $b_1,\dots,b_{p}\in \mathbb{F}_p$ with $b_{p}\neq 0$. Suppose that 
$$
\sum_{j=1}^{p-1}a_jx_j^{p-1}= \sum_{j=1}^{p}b_jx_j=0
$$
has no non-singular solution in $\mathbb{F}_p$. Then the $a_j$ are all equal.
\end{lemma}

\begin{proof} First notice that the equation $\sum_{j=1}^{p-1}a_jx_j^{p-1}=0$ has no non-trivial solution (otherwise, by following the lines of the proof of Lemma \ref{L4}, we would have a non-singular solution to the system). The lemma now follows from Lemma \ref{L2}.
\end{proof}

\begin{lemma}\label{L6} Let $p\ge 5$. Suppose that $a_1,\dots,a_{p+2}\in \mathbb{F}_p^{\times}$,  and that at least one of the $b_j\in \mathbb{F}_p$ is non-zero. Then there is a non-singular solution in $\mathbb{F}_p$ of the  equations
$$
\sum_{j=1}^{p+2}a_jx_j^{p-1}= \sum_{j=1}^{p+2}b_jx_j=0.
$$
\end{lemma}
\noindent This result  is a trivial consequence of Lemma \ref{L7} below.

\begin{lemma} \label{L7} Let $p\ge 5$. Suppose that $a_1,\dots,a_{p+1}\in \mathbb{F}_p^{\times}$, and that at least one of the $b_j\in \mathbb{F}_p$ is non-zero. Suppose that 
\be{3.0}
\sum_{j=1}^{p+1}a_jx_j^{p-1}= \sum_{j=1}^{p+1}b_jx_j=0
\ee
has no non-singular solution in $\mathbb{F}_p$. Then, after a permutation of indices, the matrix of coefficients is of the form
\be{kk1}
\left(
\begin{array}{ccccc}
a&p-a&a'&\dots & a'\\
b_1&b_2&0&\dots & 0\\
\end{array}
\right)
\ee
with $a,a',b_1,b_2\in \mathbb{F}_p^{\times}$.
\end{lemma}

\begin{proof} Suppose that exactly $t$ of the numbers $b_j$ are non-zero. Then, by renumbering indices, we may assume that $b_1\cdots b_t\neq 0$, and  $b_j=0$ for $j>t$.

We first consider the case where $t\ge 3$. Then, on applying Lemma \ref{L3} to $a_1,\dots, a_t$, we may again rearrange indices to assume that $a_1+a_2\neq 0$. We now apply Lemma \ref{L1} to find $x_3,\dots, x_{p+1}$ with
$$
\sum_{j=3}^{p+1}a_jx_j^{p-1}=-(a_1+a_2).
$$
Let 
$$
B=\sum_{j=3}^{p+1}b_jx_j.
$$
Now choose $x_1\in \mathbb{F}_p^{\times}$ such that $b_1x_1+B\neq 0$, and then $x_2\in \mathbb{F}_p^{\times}$ with $b_1x_1+b_2x_2+B=0$. This shows that $(x_1,\dots, x_{p+1})$ is a solution of \rf{3.0}.  For $y\in \mathbb{F}_p$, we put
$$
z_1=x_1+b_2y,\qquad z_2=x_2-b_1y.
$$
Then, we have $b_1z_1+b_2z_2+B=0$ irrespective of the values of $y$. Further, if $y$ is chosen such that $z_1z_2\neq 0$, we conclude that $(z_1,z_2,x_3,\dots, x_{p+1})$ is also a solution of \rf{3.0}. We claim that for some $y$ the solution is non-singular mod $p$. To see this, consider the minor
$$
\Delta_{1,2}(z_1,z_2)=\left(\begin{array}{cc}
(p-1)a_1z_1^{p-2}&(p-1)a_2z_2^{p-2}\\
b_1&b_2\\
\end{array}\right)
$$
of the Jacobian corresponding to indices $1$ and $2$. Since $z_1z_2\neq 0$, one has 
\begin{align*}
z_1z_2\det \Delta_{1,2}(z_1,z_2)&=(p-1)(a_1b_2z_2-a_2b_1z_1)\\ &=(p-1)\big(a_1b_2x_2-a_2b_1x_1-b_1b_2(a_1+a_2)y \big).
\end{align*}
Since  $p\ge 5$, one can choose $y$ such that $z_1z_2\det \Delta_{1,2}(z_1,z_2)\neq 0$. This provides the desired non-singular solution. 

Next we consider the case $t=2$. If $a_1+a_2\neq 0$, the previous argument still applies, and again yields a non-singular solution of the system \rf{3.0}. This leaves the case where $a_2=-a_1$. If one can find a non-trivial solution of 
$$
\sum_{j=3}^{p+1}a_jx_j^{p-1}=0,
$$
then take $x_1=x_2=0$ to obtain  a non-singular solution of \rf{3.0} . Hence, by Lemma \ref{L2}, all $a_j$ ($3\le j\le p+1$) are equal, which is \eqref{kk1}.
When $t=1$, take $x_1=0$ and use Lemma \ref{L1}.
\end{proof}

\section{Normalisation}\label{k=p-1/norm}

We now turn  to  solutions of the system \rf{01} in
$p$-adic numbers, and begin with a variant of a normalisation introduced by Davenport and Lewis \cite{DL63}.

Suppose we are given a system of equations \rf{01}  with rational coefficients $a_j,b_j$. Another such system is said to be {\em equivalent} to the given one if it can be transformed
into the given one by a finite succession of the following processes:

\begin{enumerate}[(i)]
\item  substitutions $(x_1, \dots, x_s)\mapsto (c_1x_1,\dots,c_sx_s)$,  with all $c_j\in \QQ^{\times}$,
\item  multiplication of one of the equations by a non-zero rational number,
\item  permutation of indices.
\end{enumerate}

This defines an equivalence relation, and if one system  \rf{01} has a non-trivial $p$-adic solution, then so have all equivalent systems.

Note that each equivalence class contains a system with integer coefficients. Further we remark that if $a_ib_i\neq 0$ holds for all $1\le i\le s$, then this is so for all equivalent systems.

A system \rf{01} with integer coefficients is referred to as {\em preconditioned} (for $p$)  if all its coefficient $a_j$, $b_j$ are non-zero, and there exists a $b_i$ with $p\nmid b_i$. A preconditioned system is said to be {\em conditioned} if for $1\le j\le k$, one has
\be{2}
\#\{1\le i\le s\colon p^{j}\nmid a_i\}\ge js/k.
\ee

\begin{lemma}\label{L8} Fix natural numbers $k$ and $s$. Suppose that for all conditioned systems \rf{01} there exists non-trivial $p$-adic solutions. Then all systems \rf{01} with rational coefficients have non-trivial $p$-adic solutions.

\end{lemma}

\begin{proof} The proof is in two steps. We first show that a system \rf{01} with rational coefficients and $a_ib_i\neq 0$ for all $1\le i\le s$ has a non-trivial $p$-adic solution. According to a comment in the preamble of Lemma \ref{L8}, 
this will follow from showing that such a system is equivalent to a conditioned system.

To see this, multiply the equations \rf{01} with a suitable natural number to arrange that $a_i,b_i$ are integers. Then define $\rho_i$ by $p^{\rho_i}\|a_i$ and write $\rho_i=\alpha_ik+\nu_i$ with $0\le \nu_i\le k-1$. We apply the transformation $x_i\mapsto p^{-\alpha_i}x_i$ for all $i$. Then the new system has $\rho_i=\nu_i$. On multiplying the linear equation by a suitable integer, the new system can still be supposed to have integer coefficients. For this system, define
\be{def-ups}
\upsilon_j=\# \{1\le i\le s\colon \nu_i=j\}
\ee
and apply a permutation of indices such that the variables with $\nu_i=0$ are numbered $1,2,\dots, \ups_0$, the variables with $\nu_i=1$ are numbered $\ups_0+1,\dots, \ups_0+\ups_1$, and so on. With ${\bf x}_0=(x_1,\dots, x_{\ups_0})$, ${\bf x}_1=(x_{\ups_0+1},\dots, x_{\ups_0+\ups_1})$ \textit{etc},  we then have
\be{3}
\sum_{i=1}^{s}a_ix_i^k=f_0({\bf x}_0)+pf_1({\bf x}_1)+\dots +p^{k-1}f_{k-1}({\bf x}_{k-1})
\ee
where
$$
f_j({\bf x}_j)=p^{-j}\sum_{\nu_i=j}a_ix_i^k
$$
has integer coefficients. Next apply the transformation ${\bf x}_0\mapsto p{\bf x}_0$,  followed by division of  \eqref{3}   by $p$. This  transforms \eqref{01} into an equivalent system where \eqref{3} now becomes
\be{4}
f_1({\bf x}_1)+p f_2({\bf x}_2)+\dots + p^{k-2} f_{k-1}({\bf x}_{k-1})+p^{k-1}f_0({\bf x}_0).
\ee

Repetition of this argument shows that any cyclic permutation of the $f_j$ is possible. Note that this also permutes the $\ups_j$ accordingly. By \cite[Lemma 2]{DL63}, there is a cyclic permutation of the  $\ups_j$ with 
\be{3ups}
\ups_0+\dots + \ups_j\ge (j+1)s/k
\ee
for $0\le j\le k-1$. Hence, the new system satisfies \rf{2}. After multiplication  by a suitable natural number, the linear equation will have  integer coefficients, and on cancelling redundant factors $p$, one obtains a conditioned system equivalent to the original one.

In a second step, we apply a compactness argument of Davenport and Lewis. If the system \rf{01} with integer coefficients has some $a_i$ or $b_i$ zero, then for all large $n\in \NN$, the numbers $a'_i=a_i+p^n$, $b'_i=b_i+p^n$ are non-zero. Thus, the system \rf{01} with $a'_i,b'_i$ in place of $a_i,b_i$ has a non-trivial $p$-adic solution ${\bf z}_n$. By homogeneity, we may suppose that ${\bf z}_n\in \ZZ_p^{s}\setminus p\ZZ_p^s$. Since $\ZZ_p^s$ is compact, the sequence $({\bf z}_n)_n$ contains a convergent subsequence. By the argument given in \cite{DL69}, page 573, its limit is a non-trivial solution of the given system. 
\end{proof}

\section{The case $k=p-1$: a reduction step}

We  require some notation that we shall use throughout the next three sections. First and foremost, we suppose that
$k=p-1$. Recall here also that we assumed that $k\ge 4$ so that $p\ge 5$. Further,  let \rf{01} be a  system with non-zero integer coefficients and $p^k\nmid a_i$ for all $1\le i\le s$. Then define the numbers
$\nu_i,\mu_i$ via
$$
p^{\nu_i}\|a_i,\quad p^{\mu_i}\|b_i.
$$
The variable $x_i$ in \rf{01} (or the index $i$) is said to be {\em low} when $\mu_i<\nu_i$, and {\em high} otherwise. The number $\min(\mu_i,\nu_i)$ is called the {\em level} of the variable $x_i$.
We now mimic some of the analysis from the proof of Lemma \ref{L8}. We define $\ups_j$ by \rf{def-ups} and note that $p^k\nmid a_i$ ($1\le i\le s$) implies that $\ups_j$ vanishes for $j\ge k$. Hence, after a suitable permutation of the variables $x_i$, the given form of degree $k$ can be represented as in \rf{3}. In particular, the vectors ${\bf x}_j$ and the forms $f_j$ are defined in our current context.  Note that  conditioned systems are covered by this set-up, and for these one  has the additional inequality \rf{3ups}.

\begin{lemma}\label{L9-1} Let \rf{01} be a  system with non-zero integer coefficients. Suppose that $\ups_0\ge k+1$, and that there is a low variable at level $0$. Then the system has a non-trivial $p$-adic solution. In particular, a conditioned system with $s\ge k^2+2$ and  a low variable at level $0$ has a non-trivial $p$-adic solution.
\end{lemma}

\begin{proof} The variables $x_1,\dots ,x_{\ups_0}$ are at level $0$, and are high by definition. Hence, if $x_j$ is a low variable at level $0$, then $j>\ups_0$ and $p\nmid b_j$. 
We take $x_i=0$ for $i>\ups_0$, except for one low variable $x_j$. Then we apply 
Lemma \ref{L4} with $x_j$ in the role of $y$ to obtain   a  non-singular solution of the pair of congruences
$$
\sum_{i=1}^sa_ix_i^k\equiv \sum_{i=1}^sb_ix_i\equiv 0\bmod p.
$$
Then, Lemma \ref{lemmaH} yields the desired $p$-adic solution of \rf{01}. For conditioned systems, the inequality $\ups_0\ge k+1$ follows from \rf{3ups}.
\end{proof}

\begin{lemma}\label{L9-2}  Let \rf{01} be a  system with non-zero integer coefficients.  Suppose that $\ups_0\ge k+3$. Then there exists a non-trivial $p$-adic solution.
\end{lemma}

\begin{proof} On cancelling redundant factors $p$ from the linear equation, we may suppose that $p\nmid b_j$ for at least one $j$. If $j>\ups_0$, then $x_j$ is low and Lemma \ref{L9-1} yields a non-\-trivial $p$-adic solution. If $j\le \ups_0$, then Lemma \ref{L6} yields a non-singular solution of 
\be{6-1}
\sum_{j=1}^{\ups_0}a_jx_j^k\equiv \sum_{j=1}^{\ups_0}b_jx_j\equiv 0\bmod p.
\ee
We may take $x_j=0$ for $j>\ups_0$ and apply
Lemma \ref{lemmaH} to find  a non-trivial $p$-adic solution of \rf{01}. 
\end{proof}

\begin{lemma}\label{L9-3} Let $s\ge k^2+2$, and suppose that the system \rf{01} is conditioned. Suppose further that for some $j\in \{1,\dots, k-1\}$ one has $\ups_j\ge k+1$. Then there exists a non-trivial $p$-adic solution.
\end{lemma}

\begin{proof}  In view of Lemma \ref{L9-1}, we may assume that no variable at level $0$ is low. Hence the variables at level $0$ are exactly $x_1,\dots, x_{\ups_0}$, and $p\mid b_m$ for all $m>\ups_0$. Since the system is conditioned, there is $i_0\le \ups_0$ with $p\nmid b_{i_0}$.
We apply ${\bf x}_i\mapsto p{\bf x}_i$ for $0\le i\le j-1$. We then divide the degree $k$ equation by $p^{j}$, and the linear equation by $p$. The  new  system has integer coefficients,  is  equivalent with the given one, and the variables in ${\bf x}_j$ are now at level $0$. Also the variable $x_{i_0}$ is a low variable at level $0$ in the new system. Hence, Lemma \ref{L9-1} yields a non-trivial $p$-adic solution.
\end{proof}

We now summarise the impact of the above lemmata on conditioned systems.

\begin{lemma}\label{L9-4} Suppose  that $s\ge k^2+2$, and that the system \rf{01} is conditioned. If this system does not have a non-trivial $p$-adic solution, then
\be{8}
s=k^2+2,\quad \ups_0=k+2,\quad \ups_j=k\quad (1\le j\le k-1),
\ee
and for all $1\le j\le k-1$, the forms $f_j$ as defined in \rf{3} satisfy
\be{8-1}
f_j(z_1,\dots, z_k)\equiv c_j(z_1^k+\dots+z_k^k)\bmod p
\ee
for some integer $c_j$ with $p\nmid c_j$.
\end{lemma}

\begin{proof} Since the system is conditioned, but does not have  a non-trivial $p$-adic solution, we deduce from Lemma \ref{L9-3} that $\ups_j\le k$ for $1\le j\le k-1$, and from  Lemma \ref{L9-2} that $\ups_0\le k+2$. But $\ups_0+\dots+\ups_{k-1}=s\ge k^2+2$, and \rf{8} follows.

Now let $j\in \{1,\dots, k-1\}$. By Lemma \ref{L9-1}, no variable at level $0$ is low. Hence, the argument of proof of Lemma \ref{L9-3}
shows that the given system is equivalent to one where the variables $x_i$ that originally had $\nu_i=j$  are now  at level $0$, and the new system has an extra low variable at level $0$. Lemma \ref{L5} is applicable to the new system, and in view of Lemma \ref{lemmaH}, we may 
conclude that the coefficients of $f_j$ are all equal, mod $p$. This gives \rf{8-1}.
\end{proof}

From now on, we are reduced to consider conditioned systems where \rf{8} holds. By Lemma  \ref{L7}, {\em either} there is   a non-singular solution of the congruences \rf{6-1}, and then via Lemma \ref{lemmaH} a 
non-trivial $p$-adic solution of \rf{01}, {\em or} there is a permutation of indices and integers $a,a', b_1,b_2$ with $p\nmid aa'b_1b_2$
and
\be{7}
\left(\begin{array}{c}
a_i\\
b_i\\
\end{array}
\right)_{1\le i\le \ups_0}\equiv \left(
\begin{array}{ccccc}
a&-a&a'&\dots & a'\\
b_1&b_2&0&\dots & 0\\
\end{array}
\right)\bmod p.
\ee
Thus, we may suppose that the conditioned system satisfies both \rf{8} and \rf{7}. We now multiply the degree $k$ equation of the given system with $b_1^kb_2^k$. Note that this does not affect the numbers $\nu_j$ because $b_1^kb_2^k\equiv 1\bmod p$. Since $p$ is odd, the substitution $x_1'=b_1x_1$, $x_2'=-b_2x_2$ takes the given system to an equivalent system where the new coefficients, say $a_j,b_j$ again, satisfy $b_1=1$, $b_2=-1$, while \rf{7} still holds. Now choose  an integer $a''$ with $a'a''\equiv 1\bmod p$ and multiply the degree $k$ equation in \rf{01} by $a''$. In this way we arrange that \rf{7} holds with $a'=1$. 
We compile this argument as the following result.

\begin{lemma}\label{L9-5} A conditioned system \rf{01} with  \rf{8} and \rf{7} is equivalent to a conditioned system satisfying $b_1=-b_2=1$ and
\be{9}
\left(\begin{array}{c}
a_i\\
b_i\\
\end{array}
\right)_{1\le i\le \ups_0}\equiv \left(
\begin{array}{ccccc}
a&-a&1&\dots & 1\\
1&-1&0&\dots & 0\\
\end{array}
\right)\bmod p.
\ee
\end{lemma}

It remains to solve conditioned systems of the shape introduced in Lemma \ref{L9-5}. If in such a system one has $a_1=-a_2$, then $x_1=x_2=1$ and $x_j=0$ ($j\ge 3$) is a non-trivial rational solution. Hence we may suppose that $a_1+a_2\neq 0$.

We now refer to a conditioned system as {\em critical} if the following conditions are satisfied:

\begin{enumerate}[(i)]
\item  $a_1+a_2\neq 0$, $b_1=-b_2=1$,
\item the equations \rf{8} hold,
\item the congruences  \rf{9} hold,
\item for $1\le j\le k-1$, the congruences \rf{8-1} hold,
\item there is no low variable at level $0$.
\end{enumerate}
In this language, Lemmata \ref{L9-1}, \ref{L9-4} and \ref{L9-5} may be summarised as follows.

\begin{lemma}\label{L9-6}Suppose that $s\ge k^2+2$ and that the conditioned system \rf{01} does not have a non-trivial $p$-adic
solution. Then, the system is equivalent to a critical system.
\end{lemma}

\section{The case $k=p-1$: critical systems}

In this and the next  section, we show that any critical system \rf{01} has non-trivial $p$-adic solutions. This is the most demanding part of our proof of the theorem. It will turn out that the variables $x_1,x_2$ can be grouped together with a block of variables, all with the same value of $\nu_j$,  to form a subsystem that is readily solved over $\QQ_p$. However, the selection process for this block depends on the distribution of the numbers $\mu_3,\dots, \mu_{k^2+2}$ in a delicate manner.

\medskip
For a critical system, the integers $a_1,a_2$ are not divisible by $p$, but we have $p\mid a_1+a_2$ and $a_1+a_2\neq 0$. Hence, there is $\theta\in \NN$ with $p^{\theta}\| a_1+a_2$. Throughout, we assume that $a_1,a_2$ have these properties and define $\theta$ even if $a_1,a_2$ are not related to a critical system.

\begin{lemma}\label{lemA} Let $a_1,a_2$ be  as in the preceding paragraph, and let $c,d$ be integers with $p\nmid cd$. Then for each   $l$ with $1\le l<\theta$, there are integers  $x_1,x_2,c'$ with $c'\equiv c\bmod p$ and 
$$
a_1x_1^k+a_2x_2^k= p^lc', \quad
x_1-x_2= p^ld.
$$     
\end{lemma}
\begin{proof} Since $k=p-1$, we see that $p\nmid k$, and by Fermat's theorem, there is a natural number $x$ with $ka_1dx^{k-1}\equiv c\bmod p$. Now choose $x_2=x$, $x_1=x+p^{l}d$. Then $a_1x_1^k+a_2x_2^k=a_1(x+p^ld)^k+a_2x^k$, and 
 we have assumed that $2\le l+1\le \theta$. Hence, $l+1\le 2l$, and it follows that
$$
a_1x_1^k+a_2x_2^k
\equiv (a_1+a_2)x^k+ka_1dx^{k-1}p^lk\equiv cp^l\bmod p^{l+1},
$$
as required.
\end{proof}

The next two lemmas are concerned with auxiliary systems that we shall meet recursively in the course of the argument.

\begin{lemma}\label{lemB}Let $a_1,a_2$ as in the preamble of Lemma \ref{lemA}. Let $c_1,\dots,c_k,d_1,\dots,d_k$, $e,f$ be integers where $p\nmid c_1f$, and where
\be{C6-1}
c_1\equiv c_2\equiv \dots \equiv c_k\bmod p.
\ee
Let $1\le \beta<\theta$. Then, the system of equations
\be{C6-2}
\begin{array}{rllr}
a_1x_1^k+a_2x_2^k&+p^{\beta}(c_1y_1^k+\dots +c_ky_k^k)&+p^{\beta+1}ez^k&=0,\\
x_1-x_2&+p^{\beta}(d_1y_1+\dots +d_ky_k)&+p^{\beta}fz&=0\\
\end{array}
\ee
has a non-trivial solution $(x_1,x_2,y_1,\dots,y_k,z)\in \QQ_p^{k+3}$.
\end{lemma}
\begin{proof} We apply Lemma \ref{lemA} with $l=\beta$, $d=1$ and $c=-kc_1$. Lemma \ref{lemA} then delivers numbers $x_1,x_2\in \ZZ$ that we insert in \rf{C6-2}. A factor $p^{\beta}$ can now be cancelled from both equations in \rf{C6-2}, and these equations now reduce to
\be{C6-3}
\begin{array}{rllr}
c'&+c_1y_1^k+\dots +c_ky_k^k&+pez^k&=0,\\
1 &+d_1y_1+\dots +d_ky_k&+fz&=0,\\
\end{array}
\ee
in which $c'$  is a  certain integer with $c'\equiv -kc_1\bmod p$. Now note that $y_1=y_2=\dots=y_k=1$ and a suitable $z\in \NN$ solve the pair of congruences
\begin{align*}
c'+c_1y_1^k+\dots + c_ky_k^k+pez^k&\equiv 0\bmod p,\\
1 +d_1y_1+\dots + d_ky_k+fz&\equiv 0\bmod p,
\end{align*}
and the Jacobian determinant associated with $y_k$ and $z$ at this solution is 
$$
k\big(y_k^{k-1}c_kf-pez^{k-1}d_k\big)\equiv kc_kf\not\equiv 0\bmod p.
$$
Consequently, Lemma \ref{lemmaH} provides a solution of \rf{C6-3} in $p$-adic numbers in which $y_k\neq 0$. This solution, together with the $x_1,x_2$ chosen earlier, is a solution of \rf{C6-2}.
\end{proof}

\begin{lemma}\label{lemB'} Let $a_1,a_2$ be as in the preamble of Lemma \ref{lemA}. Let $c_1,\dots,c_k,d_1,\dots,d_k$ be integers with $p\nmid c_1d_1$  and \rf{C6-1}. Let $1\le \beta<\theta$. Then, the system of equations
\be{C6-4}
\begin{array}{rll}
a_1x_1^k+a_2x_2^k&+p^{\beta}(c_1y_1^k+\dots +c_ky_k^k)&=0,\\
x_1-x_2&+p^{\beta}(d_1y_1+\dots +d_ky_k)&=0\\
\end{array}
\ee
has a non-trivial $p$-adic solution.
\end{lemma}
\begin{proof} 
Write $d_2=p^{m}d'_2$ with $p\nmid d'_2$. Put  $u=1-p^{m}$, so that $u=0$ when $p\nmid d_2$, and 
$u\equiv 1\bmod p$ otherwise.   Then apply Lemma \ref{lemA} with $c=-c_1- u^kc_2$,  $d=-d_1$ and $l=\beta$. Note that $p\nmid cd$ as required. This lemma provides integers $x_1,x_2,c'$ with $c'\equiv c\bmod p$. If we  take $y_3=y_4=\dots=y_k=0$  in \rf{C6-4} and cancel a factor $p^{\beta}$, this system now  reduces to
\be{C6-5}
\begin{array}{rl}
c'&+c_1y_1^k+c_2y_2^k=0,\\
-d_1&+d_1y_1+p^{m}d'_2y_2=0.\\
\end{array}
\ee
By construction, the pair $y_1=1$, $y_2=u$ is a solution of the congruences
\begin{align*}
c'+c_1y_1^k+c_2y_2^k&\equiv 0\bmod p,\\
-d_1+d_1y_1+p^{m}d'_2y_2&\equiv 0\bmod p,
\end{align*}
and the Jacobian determinant at this solution is $k(c_1p^m d'_2-c_2d_1u^{k-1})$. By considering separately the cases $m=0$ and $m\neq 0$, one observes that this                               determinant is  not divisible by $p$. As in the proof of Lemma \ref{lemB}, a non-trivial solution of \rf{C6-5} in $\QQ_p$ is now supplied by Lemma \ref{lemmaH}, and this unfolds to such a solution of \rf{C6-4}.     
\end{proof}

The next two results are consequences  of the last two lemmata for critical systems.
\begin{lemma}\label{lemC} 
Suppose that \rf{01} is a critical system, and that there is a low variable at level $l$ with $l<\theta$. Then the system \rf{01} has a non-trivial $p$-adic solution.
\end{lemma}
\begin{proof} Recall that critical systems have no low variables at level $0$. Now consider all low variables and choose one, say $x_t$, where the level $\beta$ is the smallest among them.
Then $1\le \beta<\theta$. Further, the variables ${\bf x}_{\beta}$ of $f_{\beta}({\bf x}_{\beta})$ are all high, thanks to the minimality of $\beta$. We put all variables in \rf{01} to $0$ except $x_1, x_2,{\bf x}_{\beta}$ and $x_t$. With ${\bf x}_{\beta}=(y_1,\dots,y_k)$, $x_t=z$ and $a_t=p^{\beta+1}e$, $b_t=p^{\beta}f$, we have  $e,f\in \ZZ$ with $p\nmid f$, and the system \rf{01} reduces to the system \rf{C6-2}, with the conditions for application of Lemma \ref{lemB} satisfied.
This yields the desired solution of \rf{01}.
\end{proof}

\begin{lemma}\label{lemC'}
Suppose that \rf{01} is a critical system that  involves a variable $x_t$ with $1\le \nu_t=\mu_t<\theta$. Then the system has a non-trivial $p$-adic solution.
\end{lemma}
\begin{proof}
The variable $x_t$ is at level $\beta=\nu_t$, and therefore occurs among the entries of ${\bf x}_{\beta}=(y_1,\dots,y_k)$, say. By symmetry, we may suppose that $x_t=y_1$. If any of the variables $y_2,\dots,y_k$ were low, then Lemma \ref{lemC} would supply the desired solution of \rf{01}, so we may suppose that $y_2,\dots,y_k$ are all high. We take all $x_j$ in \rf{01} as $0$ except $x_1,x_2$ and ${\bf x}_{\beta}=(y_1,\dots,y_k)$. Then \rf{01} reduces to the system \rf{C6-4}, with the conditions for applicability of Lemma \ref{lemB'} all met. This yields the desired solution of \rf{01}.
\end{proof}

We now establish a result that complements  Lemmas \ref{lemC} and \ref{lemC'}. The strategy is different from the above approach. In particular, we rely on the classical version of Hensel's Lemma, and contract the variables $x_1$ and $x_2$ suitably.

\begin{lemma}\label{lemD}
Suppose that \rf{01} is a critical system. Write $\theta=\ups k+r$ with $0\le r\le k-1$. For all $i\ge 3$ with $\nu_i=r$ suppose that $\mu_i>\theta-\ups$ holds. Then \rf{01} has a non-trivial $p$-adic solution.
\end{lemma}

\begin{proof}Recall that for a critical system the variables $x_i$ with $i\ge 3$ and $\nu_i=r$ are exactly those where $rk+2<i\le rk+k+2$. For convenience, we put 
${\bf y}=(x_{rk+3},\dots,x_{rk+k+2})$ and then set all variables in \rf{01} to $0$ except $x_1,x_2$ and ${\bf y}$. Renaming coefficients, the system \rf{01} then reduces to the pair of equations
$$
\begin{array}{rl}
a_1x_1^k+a_2x_2^k&+p^{r}(c_1y_1^k+\dots +c_ky_k^k)=0,\\
x_1-x_2&+p^{\theta-\ups+1}(d_1y_1+\dots +d_ky_k)=0\\
\end{array}
$$
in which $c_i,d_i$ denote integers with $p\nmid c_i$ and \rf{C6-1}. We put ${\bf y}=p^{\ups}{\bf z}$. Then the system becomes
\be{C6-7}
\begin{array}{rl}
a_1x_1^k+a_2x_2^k&+p^{\theta}(c_1z_1^k+\dots +c_kz_k^k)=0,\\
x_1-x_2&+p^{\theta+1}(d_1z_1+\dots +d_kz_k)=0,\\
\end{array}
\ee
and it now suffices to construct a non-trivial $p$-adic solution of this pair of equations.

Write $a_1+a_2=p^{\theta}a'$. Then $a'\in \ZZ$ with $p\nmid a'$. By Lemma \ref{L1}, we can choose integers $z_1,\dots,z_k$ with $c_1z_1^k+\dots +c_kz_k^k\equiv -a'\bmod p$. Not all of the $z_i$ can be divisible by $p$, and by symmetry, we may suppose that $p\nmid z_1$. With these integers determined, put
\be{C6-8}
h=-p^{\theta+1}(d_1z_1+\dots +d_kz_k).
\ee
With a variable $x\in \QQ_p$ still at our disposal, we choose
$$
x_2=x,\quad x_1=x+h,
$$
and substitute in \rf{C6-7}. Then, the linear equation of \rf{C6-7} is satisfied irrespective of the value of $x$. Further, the first equation in \rf{C6-7} reduces to
\be{C6-9}
a_1(x+h)^k+a_2x^k - p^{\theta}c=0
\ee
where according to our construction, the integer $c=-(c_1z_1^k+\dots+c_kz_k^k)$ satisfies $c\equiv a'\bmod p$. However
\be{C6-10}
a_1(x+h)^k+a_2x^k=p^{\theta}a'x^k+ka_1x^{k-1}h+h^2Q_k(x,h)
\ee
where $Q_k\in \ZZ[x,h]$ is a certain polynomial. With $h$ fixed via \rf{C6-8}, it follows that
$$
\phi(x)=p^{-\theta}\big( a_1(x+h)^k+a_2x^k  \big)
$$
is a polynomial with integer coefficients, and from \rf{C6-8} and \rf{C6-10} we see that $\phi(1)\equiv a'\bmod p$ and $\phi'(1)\equiv ka'\not\equiv 0\bmod p$. Hence $x=1$ is a solution of the congruence $\phi(x)-c\equiv 0\bmod p$. By Hensel's Lemma, there is a non-zero $x\in \QQ_p$ with $\phi(x)-c=0$, and this $x$ also solves \rf{C6-9}. This completes the proof of Lemma \ref{lemD}.
\end{proof}

  \begin{lemma} \label{lem811} The conclusion of Lemma \ref{lemB'} remains valid when $\beta=\theta$.
  \end{lemma}

\begin{proof} We recast the system \eqref{C6-4} that now takes the shape
\be{C6-12}
\begin{array}{rl}
a_1x_1^k+a_2x_2^k&+p^{\theta}(c_1y_1^k+\dots +c_ky_k^k)=0,\\
x_1-x_2&+p^{\theta}(d_1y_1+\dots +d_ky_k)=0,\\
\end{array}
\ee
in  which $c_i,d_i$ are certain integers with \rf{C6-1} and $p\nmid c_1$, and not all the $d_i$ are divisible by $p$. From now on, we assume that $c_i\equiv 1\bmod p$ holds for all $1\le i\le k$, and we may do so without  loss  of generality. To see this, choose $c\in \NN$ with $cc_1\equiv 1\bmod p$, and multiply the top equation in \rf{C6-12} by $c$. Then, we still have $p^{\theta}\|ca_1+ca_2$, and \rf{C6-1} implies that $cc_j\equiv 1\bmod p$ for all $j$, as required.

By symmetry, we may further suppose that 
$$
p\nmid d_i\quad (1\le i\le i_0),\qquad p\mid d_i\quad (i_0<i\le k)
$$
holds with a suitable number $i_0\in \{1,\dots,k\}$.

We now discuss the equations \rf{C6-12} by a blend of ideas now familiar from the proofs of Lemmas \ref{lemC} and \ref{lemD}. Put $a_1+a_2=p^{\theta}a'$, and write $a'\equiv -\alpha\bmod p$ with $1\le \alpha\le p-1$.

There will be four cases to consider.

\medskip
(i) Suppose that $\alpha\ge 2$ and $i_0\ge 2$. Then we take $x_1=x_2=1$ and $y_j=0$ $(\alpha<j\le k)$ in \rf{C6-12} which then reduces to 

\be{C6-13}
\begin{array}{rl}
a'+&c_1y_1^k+\dots+c_{\alpha}y_{\alpha}^k=0\\
&d_1y_1+\dots +d_{\alpha}y_{\alpha}=0.\\
\end{array}
\ee
But since $\alpha\ge 2$ and $i_0\ge 2$ hold simultaneously it is immediate that there exist integers $z_1,\dots,z_{\alpha}$, all not divisible by $p$, with $d_1z_1+\dots+d_{\alpha}z_{\alpha}\equiv 0\bmod p$. Since all $c_j$ are in the class $1\bmod p$, it follows that the congruences
\be{C6-14}
\begin{array}{rl}
a'+&c_1z_1^k+\dots+c_{\alpha}z_{\alpha}^k\equiv 0\bmod p\\
&d_1z_1+\dots +d_{\alpha}z_{\alpha}\equiv 0\bmod p\\
\end{array}
\ee
hold simultaneously. But $\alpha<p$, and hence, the numbers $d_iz_i$ cannot all be equal, modulo $p$. Hence, we can choose $1\le i<j\le \alpha$ with $d_iz_i\not\equiv d_jz_j\bmod p$. Let $\Delta_{i,j}$ be the Jacobian determinant for $z_i,z_j$ at this solution of \rf{C6-14}. Then
$$
z_iz_j\Delta_{ij}=\det\left(
\begin{array}{cc}
kc_iz_i^{k}&kc_jz_j^{k}\\
d_iz_i&d_jz_j\\
\end{array}
\right)\equiv k(d_jz_j-d_iz_i)\bmod p,
$$
so that the solution in \rf{C6-14} is non-singular. By Lemma \ref{lemmaH},  we infer that \rf{C6-13} has a non-trivial $p$-adic solution, as required.

\medskip
(ii) Suppose that $\alpha=1$ and $i_0\ge 2$. First choose $d\in \NN$ with $a'+ka_1d\equiv -2\bmod p$. Note that this implies that $p\nmid d$. Now take $x_2=1$ and $x_1=1+dp^{\theta}$ in \rf{C6-12} as well as $y_3=\dots=y_k=0$. Then, since we have
$$
a_1x_1^k+a_2x_2^k=p^{\theta}(a'+a_1dk)+p^{2\theta}E
$$
with some $E\in \ZZ$, the equations \rf{C6-12} reduce to 
\be{C6-15}
\begin{array}{rl}
a'+a_1dk+p^{\theta}E+&c_1y_1^k+c_{2}y_{2}^k=0\\
d+&d_1y_1+d_{2}y_{2}=0.\\
\end{array}
\ee
However, there are integers $z_1,z_2$ with $p\nmid z_1z_2$ and $d_1z_1\equiv d\bmod p$, $d_2z_2\equiv -2d\bmod p$. Then, according to our choice of $d$, the numbers $z_1,z_2$ solve the congruences 
\be{C6-16}
\begin{array}{rl}
a'+a_1dk+p^{\theta}E+&c_1z_1^k+c_{2}z_{2}^k\equiv 0\bmod p,\\
d+&d_1z_1+d_{2}z_{2}\equiv 0\bmod p.\\
\end{array}
\ee
Note that the Jacobian determinant at the solution $z_1,z_2$ is not divisible by $p$. It follows from Lemma \ref{lemmaH} that \rf{C6-15} has a non-trivial solution in $\QQ_p$, as required.

\medskip
(iii) Suppose that $i_0=1$ and $\alpha\le p-2$. This is similar to case (i). Take $x_1=x_2=1$ in \rf{C6-12} which then  reduces to 
\be{C6-17}
\begin{array}{rl}
a'+&c_1y_1^k+\dots+c_{k}y_{k}^k=0,\\
&d_1y_1+\dots +d_{k}y_{k}=0.\\
\end{array}
\ee
The integers $z_1=0$, $z_2=\dots =z_{\alpha+1}=1$ and $z_j=0$ for $j\ge \alpha+2$ provide a solution of the associated congruences 
$$
\begin{array}{rl}
a'+&c_1z_1^k+\dots+c_{k}z_{k}^k\equiv 0\bmod p,\\
&d_1z_1+\dots +d_{k}z_{k}\equiv 0\bmod p.\\
\end{array}
$$
Further, the Jacobian determinant with respect to $z_1,z_2$ is not divisible by $p$ (note here that $z_1=0$ and $p\mid d_2$). Once again via Lemma \ref{lemmaH}, this yields a non-trivial $p$-adic solution of \rf{C6-17}.

\medskip
(iv) Suppose that $i_0=1$ and $\alpha=p-1$. We choose $d,x_1$ and $x_2$ as in case (ii), and also put $y_3=\dots =y_k=0$. We then again reduce to the system \rf{C6-15}, but this time with $p\nmid d_1$, $p\mid d_2$. Choose $z_1$ with $d_1z_1\equiv -d\bmod p$. Then $p\nmid z_1$. Also, take $z_2=1$. Then, by construction, \rf{C6-16} holds, with Jacobian determinant not divisible by $p$. As in case (ii), one is led to a non-trivial $p$-adic solution.  This completes the proof. 
\end{proof}

We are ready to treat all critical systems with small $\theta$.

\begin{lemma}\label{lemE}
A critical system with $\theta<k$ has non-trivial $p$-adic solutions.
\end{lemma}
 
\begin{proof}Recall that $\theta\ge 1$, and hence that all variables $x_i$ with $\nu_i=\theta$ are those where
\be{C6-11}
\theta k+3\le i\le \theta k+k+2.
\ee

First suppose that for all $i$ in \rf{C6-11} one has $\mu_i>\theta$. Then Lemma \ref{lemD} yields the desired $p$-adic solution.

Next suppose that there is an $i$ as in \rf{C6-11} where $\mu_i<\theta$. Then $x_i$ is a low variable at a level less than $\theta$. In this case Lemma \ref{lemC} provides a non-trivial $p$-adic solution.

In the cases not yet considered one has $\mu_i\ge \theta$ for all $i$ in \rf{C6-11}, and $\mu_i=\theta$ for at least one of the $i$ in \rf{C6-11}. We now take all  $x_j=0$ in the given critical system except for $x_1,x_2$ and ${\bf y}=(x_{\theta k+3},\dots, x_{\theta k +k+2})$. Renaming coefficients, the system then takes the shape 
\eqref{C6-12}
in which $c_i,d_i$ are certain integers with \rf{C6-1} and $p\nmid c_1$, and not all the $d_i$ are divisible by $p$. The desired $p$-adic solution is now provided by Lemma \ref{lem811}.
\end{proof}

\section{The case $k=p-1$: le coup de gr\^ace}

In this section we complete our analysis of critical systems by establishing the following complement to Lemma \ref{lemE}.

\begin{lemma}\label{lemF} 
A critical system with $\theta\ge k$ has non-trivial $p$-adic solutions.
\end{lemma}

Once this lemma is established, we conclude via Lemma \ref{lemE} that all critical systems have non-trivial $p$-adic solutions. As mentioned earlier, it now follows via  Lemma \ref{L9-6}  that all conditioned systems have such solutions, and then via Lemma \ref{L8}, this finally establishes the case $k=p-1$ of the theorem.

\medskip
Given a critical system with $\theta\ge k$, we open the endgame by re-grouping its variables into blocks
\be{C7-1}
{\bf y}_j=(x_{kj+3},x_{kj+4},\dots, x_{kj+k+2})\qquad (0\le j\le k-1),
\ee
and may then present the system as 
\begin{align}\label{C7-2}
A(x_1,x_2,{\bf y}_{0},\dots,{\bf y}_{k-1})=\sum_{i=1}^{k^2+2}a_ix_i^k,
\quad
B(x_1,x_2,{\bf y}_{0},\dots, {\bf y}_{k-1})=\sum_{i=1}^{k^2+2}b_ix_i.
\end{align}
Recalling that $\theta \ge k \ge 4$, either Lemma \ref{lemC} or Lemma \ref{lemC'} will solve the system $A=B=0$ over $\QQ_p$ unless the inequalities 
\be{C7-3}
\mu_i>\nu_i
\ee
hold for all $i\ge 3$, as we henceforth assume. In this situation, we apply a transformation to the given system that we now introduce.

Let $\tau$ be a non-negative integer, and write $\tau=uk+\rho$ with $0\le \rho\le k-1$ and $u\in \ZZ$. Then define the new forms 
\be{C7-4}
\begin{array}{ll}
A_{\tau}&=A(x_1,x_2,p^{u+1}{\bf y}_{0},\dots,p^{u+1}{\bf y}_{\rho},p^u{\bf y}_{\rho+1},\dots,  p^u{\bf y}_{k-1})\\
B_{\tau}&= B(x_1,x_2,p^{u+1}{\bf y}_{0},\dots,p^{u+1}{\bf y}_{\rho},p^u{\bf y}_{\rho+1},\dots,  p^u{\bf y}_{k-1}).\\
\end{array}
\ee
Hence the systems $A_{\tau}=B_{\tau}=0$ are all equivalent with the given system, so that it suffices to find a non-trivial $p$-adic solution of one of them. With applications in mind, we write $A_{\tau}$ and $B_{\tau}$ with coefficients as 
$$
A_{\tau}({\bf x})=\sum_{i=1}^{k^2+2}a_i^{(\tau)}x_i^k,\quad B_{\tau}({\bf x})=\sum_{i=1}^{k^2+2}b_i^{(\tau)}x_i,
$$
and then introduce the numbers $\nu_i^{(\tau)}$, $\mu_i^{(\tau)}$ for $i\ge 3$ via
$$
p^{\nu_i^{(\tau)}}\|  a_i^{(\tau)},\qquad p^{\mu_i^{(\tau)}}\|  b_i^{(\tau)}.
$$
By \rf{C7-4}, one has 
\be{C7-5}
\begin{array}{lll}
\nu_i^{(\tau)}=\nu_i+k(u+1),\qquad &\mu_i^{(\tau)}=\mu_i+u+1\qquad &(3\le i\le \rho k+k+2)\\
\nu_i^{(\tau)}=\nu_i+ku,\qquad &\mu_i^{(\tau)}=\mu_i+u\qquad &( i\ge \rho k+k+3 ).\\
\end{array}
\ee

In particular, it follows that $\nu_3^{(\tau)}>\mu_3^{(\tau)}$ holds for all large $\tau$. Therefore, there is a well-defined smallest number $t$ among those $\tau$ for which there exists an index $i\ge 3$ with $\nu_i^{(\tau)}\ge \mu_i^{(\tau)}$.

There is a curious dichotomy in the argument at this point. We first consider the case $t>\theta-k$. Let $\theta=\ups k+r$,   with $0\le r<k$. By the definition of $t$, we have $\nu_i^{(\theta-k)}<\mu_i^{(\theta-k)}$ for all $i$. However, by \rf{C7-5}, when $x_i$ belongs to ${\bf y}_r$ (that is $kr+3\le i\le kr+k+2$), one has $\nu_i=r$ and
\begin{align*}
\nu_i^{(\theta-k)}=\nu_i+\ups k=r+\ups k=\theta,\qquad
\mu_i^{(\theta-k)}=\mu_i+\ups.
\end{align*}
It follows that $\mu_i+\ups>\theta$ for all $i$ with $x_i$ in ${\bf y}_r$. Hence, by Lemma \ref{lemD}, the system $A=B=0$ has a non-trivial $p$-adic solution.

\medskip
It remains to consider the case where $t\le \theta-k$. We put $t=u'k+\rho'$. There is at least one index $i$ with $\mu_i^{(t)}\le \nu_i^{(t)}$, and thanks to the minimality of $t$, the variable $x_i$ must belong to ${\bf y}_{\rho'}$. This follows from \rf{C7-5}. Further, this argument also shows that all indices $i$ with $\mu_i^{(t)}\le \nu_i^{(t)}$ belong to ${\bf y}_{\rho'}$, that is
$\rho'k+3\le i\le \rho'k+k+2$, and we can define 
$$
\beta=\min \{ \mu_i^{(t)}\colon \mu_i^{(t)}\le \nu_i^{(t)}     \}=\min \{\mu_i^{(t)}\colon \rho'k+3\le i\le \rho'k+k+2    \}.
$$
Note that in this interval for $i$ we have $ \nu_i^{(t)}=\rho'+ku'+k=t+k$ so that $\beta\le t+k\le \theta$.

First suppose that $ \mu_i^{(t)}<\nu_i^{(t)} $ holds for at least one $i$, which is the case $\beta<t+k$. Then $\beta<\theta$.  Choose  an $i'$ with $\mu_{i'}^{(t)}=\beta$ and $\rho'k+3\le i'\le \rho'k+k+2$. By the minimality of $t$, we have $\mu_{i'}^{(t-1)}\ge \nu_{i'}^{(t-1)}$. However, by \rf{C7-5},
$$
\nu_{i'}^{(t)}=\nu_{i'}^{(t-1)}+k,\quad  \mu_{i'}^{(t)}=\mu_{i'}^{(t-1)}+1
$$
so that $ \nu_{i'}^{(t)}-k\le  \mu_{i'}^{(t)}-1$, which implies that $t+k-\beta\le k-1$.
Put $\beta=u''k+\rho''$.
Recalling that $\beta<t+k$ in the case under consideration, we see that  $\rho'$ and $\rho''$ are distinct. We now consider the system $A_t=B_t=0$ in the variables
$x_1,x_2,{\bf y}_{\rho''}$ and $x_{i'}$, and put all other variables to $0$. Then, by \rf{C7-5} we see that the system reduces to a pair of equations \rf{C6-2} if we put $z=x_{i'}$  and 
${\bf y}_{\rho''}=(y_1,\dots,y_k)$. Since $\beta<\theta$, the conditions of Lemma \ref{lemB} are all met, and that lemma provides the desired non-trivial $p$-adic solution.

Hence we are now reduced to the case $\beta=t+k$ where one has $\nu_i^{(t)}\le \mu_i^{(t)}$ for all $x_i$ in ${\bf y}_{\rho'}$, with equality for at least one $i$. Then, we again consider
$A_t=B_t=0$, this time in the variables $x_1,x_2, {\bf y}_{\rho'}=(y_1,\dots,y_k)$, with all other variables set to $0$. The reduced system takes the shape \rf{C6-4}, and when $\beta=t+k<\theta$ holds, all conditions in Lemma \ref{lemB'} are met. This lemma then supplies a non-trivial $p$-adic solution of $A_t=B_t=0$.

This leaves the case $\beta=\theta$ for consideration where we have to solve \rf{C6-4} with $\beta=\theta$, subject to the conditions in Lemma \ref{lemB'}.  By Lemma \ref{lem811},  this system  always has a non-trivial $p$-adic solution. The proof of Lemma \ref{lemF} is complete.

\section{Further preparations}\label{section4-1}

With the case $k=p-1$ now settled to our satisfaction, we may concentrate on degrees of the form $k=p^{\tau}(p-1)$ with $\tau\ge 1$.   We wish to construct, via Lemma \ref{lemmaH}, a  $p$-adic solution of a conditioned system  $A({\bf x})= B({\bf x})=0$  given by \rf{3-2}, and hence  seek for a non-singular solution of the pair of congruences
\be{P4-2}
\sum_{j=1}^{\ups}a_jx_j^k\equiv 0\bmod p^{\gamma},\qquad \sum_{j=1}^{\ups}b_jx_j\equiv 0\bmod p
\ee
in which $\ups=\ups_0+\ups_1+\dots+\ups_{\gamma-1}$. There is then a dichotomy in the argument; conditioned systems fall into two classes that call for separate treatment. Thus, we refer to a conditioned system as having type A when $p\mid b_i$  for all $i>\ups_0$, and to the remaining systems as having type B. Note that the discriminating property is whether or not the variables at level $0$ are exactly those indexed by $1\le j\le \ups_0$.

\begin{lemma}\label{lemP4-1}
{\rm (a)}
 A solution of the congruences \rf{P4-2} associated with a system of type A is non-singular  whenever there is a pair $i,j$ with $1\le i,j\le \ups_0$ and $p\nmid x_j$, $p\mid x_i$, $p\nmid b_i$. \\
{\rm (b)} A solution of the congruence \rf{P4-2} associated with a system of type B is non-singular whenever there is a number $j$ with $1\le j\le \ups_0$ and $p\nmid x_j$.

\end{lemma}

\begin{proof} Consider the Jacobian matrix for the system \rf{3-2}. Its minor with respect to columns indexed by $i,j$ is 
$$
\left(
\begin{array}{cc}
ka_ix_i^{k-1}&ka_jx_j^{k-1}\\
b_i&b_j\\
\end{array}
\right).
$$
In case (a), we take the distinguished indices $i,j$. In case (b), we choose $i$ such that $x_i$ is low at level $0$. The lemma is now immediate.
\end{proof}
In the next two sections we dispose of the case where
\be{P4-1}
k=p(p-1),\quad p\mbox{ odd}.
\ee
Here,  the pivotal step  is encapsulated in the next lemma. It can be thought of as a version of Lemma  \ref{L4}   when the modulus is $p^2$.

\begin{lemma}\label{lemP4-2}Suppose that $k$ is given by \rf{P4-1}. Let $1\le t\le u$ and $u\ge p^2+2$. Further, let $c_1,\dots,c_u,d_1,\dots,d_t$ denote integers not divisible by $p$. Then the pair of congruences
\be{P4-3}
c_1x_1^k+\dots+c_ux_u^k\equiv 0\bmod p^2,\quad d_1x_1+\dots+d_tx_t\equiv 0\bmod p
\ee
has a non-singular solution in integers $x_1,\dots,x_u$.
\end{lemma}
\begin{proof}When $t=1$ or $2$, apply Lemma \ref{lemCD} with $q=p^2$  to find a non-empty set $J\subset \{3,\dots, p^2+2\}$ with 
$$
\sum_{j\in J}c_j\equiv 0\bmod p^2.
$$
Then put $x_j=1$ for $j\in J$ and $x_i=0$ for $1\le i\le u$, $i\not\in J$. This is a solution of \rf{P4-3}, and for $j\in J$ one finds that
\be{P4-4}
\det \left(
\begin{array}{cc}
c_1x_1^{k-1}&c_jx_j^{k-1}\\
d_1&d_j
\end{array}
\right)=-d_1 c_jx_j^{k-1}
\ee
is not divisible by $p$. Hence, this solution of \rf{P4-3} is non-singular.

We may now suppose that $t\ge 3$. Then, by Lemma \ref{L3}, we can rearrange indices $1,2,3$ to ensure that $p\nmid c_2+c_3$.
Then again by Lemma \ref{lemCD}, there is a set $I\subset \{4,\dots,p^2+2\}$ with
$$
\sum_{i\in I}c_i\equiv -c_2-c_3\bmod p^2.
$$
Let
$$
D=-\multsum{i\in I}{i\le t}d_i.
$$
Now choose integers $x_2,x_3$ with $d_2x_2+d_3x_3\equiv D\bmod p$ and $p\nmid x_2x_3$. It is immediate that there are at least $p-2$ (and hence at least one) such pairs with $1\le x_i\le p-1$. Then, choosing $x_i=1$ for $i\in I$ and $x_l=0$ for $l=1$   and $4\le l\le u$ with $l\not\in I$, we have a solution of \rf{P4-3} and can use  \rf{P4-4} with $j=3$ to confirm that the solution is non-singular.
\end{proof}

\section{The case $k=p(p-1)$: type A}\label{section4-2}

In this section, we discuss systems of type A when $k=p(p-1)$, $p$ an odd prime, and $s\ge k^2+2$. We shall show that in this situation, the congruences \rf{P4-2} have a non-singular solution. With this end in view, we take $x_j=0$ for all $j>\ups_0+\ups_1$ and then recall that type A system have $p\mid b_j$ for all $j>\ups_0$. It will be convenient to put $y_j=x_{\ups_0+j}$ and $c_j=a_{\ups_0+j}/p$ for $1\le j\le \ups_1$. In this notation, the congruences \rf{P4-2} read
\be{P4-5}
\begin{array}{lll}
a_1x_1^k+\dots+a_{\ups_0}x_{\ups_0}^k&+p\big(c_1y_1^k+\dots+c_{\ups_1}y_{\ups_1}^k\big)&\equiv 0\bmod p^2,\\
b_1x_1+\dots +b_{\ups_0}x_{\ups_0}&    &\equiv 0\bmod p.\\
\end{array}
\ee
Since this pair of congruences is associated with a system of type A at least one of $b_j$ with $1\le j\le \ups_0$ is not divisible by $p$. We may then suppose that $p\nmid b_1$, say. Finally, since the system is conditioned, we have the inequalities
\be{P4-6}
\ups_0\ge k+1,\qquad \ups_0+\ups_1\ge 2k+1.
\ee
If $\ups_0\ge p^2+2$, then Lemma \ref{lemP4-2} delivers a non-singular solution of \rf{P4-5} with $y_1=\dots =y_{\ups_1}=0$. Hence, from now on, we may suppose that $\ups_0\le p^2+1$. Then, by \rf{P4-6}, 
\be{P4-7}
\ups_1\ge 2k-p^2=p^2-2p.
\ee
We take $x_1=0$, and note that $\ups_0-1\ge k\ge p+3$ for $p\ge 3$. Hence, when $p\ge 5$ and not all of $b_2,\dots,b_{\ups_0}$ are divisible by $p$, Lemma \ref{L6}  yields numbers $x_2,\dots,x_{\ups_0}$, not all divisible by $p$, with
\be{P4-8}
\begin{array}{ll}
a_2x_2^k+\dots+a_{\ups_0}x_{\ups_0}^k&\equiv 0\bmod p,\\
b_2x_2+\dots +b_{\ups_0}x_{\ups_0}&    \equiv 0\bmod p.
\end{array}
\ee
When $p=3$, then $\ups_0\ge 7$, and a theorem of Olson \cite[(1)]{O} supplies a non-empty subset $J\subset \{2,\dots,\ups_0\}$ with 
$$
\sum_{j\in J}a_j \equiv {\sum_{j\in J} b_j}\equiv 0\bmod p.
$$ 
Again, this gives a solution of \rf{P4-8}, with $x_j=1$ for $j\in J$, and $x_j=0$ for the remaining $j$. Finally, when all $b_j$ are divisible by $p$, then a non-trivial solution of \rf{P4-8} is provided by Lemma  \ref{L1}. We have now shown that for all $p\ge 3$, the congruences \rf{P4-8} have a non-trivial solution, and with one such solution $x_2,\dots,x_{\ups_0}$ fixed, we define the integer $c$ through the equation
\be{P4-9}
a_2x_2^k+\dots+a_{\ups_0}x_{\ups_0}^k=cp.
\ee
In this notation, the pair of congruences \rf{P4-5} reduces to the single congruence
$$
c+c_1y_1^k+\dots+c_{\ups_1}y_{\ups_1}^k\equiv 0\bmod p.
$$
By \rf{P4-7}, we have $\ups_1\ge p$, and hence, by Lemma \ref{L1}, this congruence has a solution whenever $p\nmid c$, while in the case where $p\mid c$, we may take $y_1=\dots=y_{\ups_1}=0$.

We have now found a solution of the congruences \rf{P4-5} with $x_1=0$ and $p\nmid x_j$ for some $j\in \{2,\dots,\ups_0\}$.  In view of Lemma \ref{lemP4-1}, it follows that \textit{all conditioned systems of type A and degree $p(p-1)$ have a non-trivial $p$-adic solution.}

\section{The case $k=p(p-1)$: type B}

In this section we complete the discussion of the case $k=p(p-1)$ by considering systems of type B. As in the previous section, we take $x_j=0$ for all $j>\ups_0+\ups_1$, define $c_j$ and $y_j$ as in Section \ref{section4-2} and also put $d_j=b_{j+\ups_0}$ ($1\le j\le \ups_1$). Then the congruences \rf{P4-2} reduce to
\be{P4-10}
\begin{array}{lll}
a_1x_1^k+\dots+a_{\ups_0}x_{\ups_0}^k&+p\big(c_1y_1^k+\dots+c_{\ups_1}y_{\ups_1}^k\big)&\equiv 0\bmod p^2,\\
b_1x_1+\dots +b_{\ups_0}x_{\ups_0}&+ \thinspace d_1y_1+\dots+d_{\ups_1}y_{\ups_1}    &\equiv 0\bmod p.\\
\end{array}
\ee
We shall show that this pair has a solution with one of $x_1,\dots,x_{\ups_0}$ not divisible by $p$. This is then also a solution of \rf{P4-2}
with this property, and from Lemma \ref{lemP4-1} (b), we may then conclude that systems of this type have non-trivial solutions in $\QQ_p$.

Our method  of solving \rf{P4-10} follows the pattern of Section  \ref{section4-2} as close as is possible. Thus, when $\ups_0\ge p^2+2$, Lemma \ref{lemP4-2} yields a non-trivial solution of \rf{P4-10} with $y_1=\dots=y_{\ups_1}=0$ provided that the $b_j$ are not all divisible by $p$. In the contrary case where all $b_j$ are divisible by $p$, we find a non-trivial solution of $a_1x_1^k+\dots+a_{\ups_0}x_{\ups_0}^k\equiv 0\bmod p^2$ from Lemma \ref{lemCD} with $q=p^2$, and again we may take $y_1=\dots=y_{\ups_1}=0$.

Hence, we are again reduced to the case $\ups_0\le p^2+1$, and we may then suppose that \rf{P4-6} and \rf{P4-7} hold. We now deal with this situation in an \textit{ad hoc} manner when $p\ge 5$, and only later refine the argument when $p=3$. We begin as with systems of type A and choose a non-trivial solution of \rf{P4-8} and insert this into \rf{P4-10}. Then, for a suitable $c\in \ZZ$, the congruences \rf{P4-10} reduce to the pair
\be{P4-11}
\begin{array}{rl}
c+c_1y_1^k+\dots+c_{\ups_1}y_{\ups_1}^k\equiv 0\bmod p,\\
d_1y_1+\dots+d_{\ups_1}y_{\ups_1}\equiv 0\bmod p.\\
\end{array}
\ee
At this point, the treatment of systems of type A was simpler because then one would have $p\mid d_j$ for all $1\le j\le \ups_1$, in which case the linear congruence in \rf{P4-11} is automatically satisfied. We therefore proceed to remove the linear congruence by a contraction method similar to the one used in Section  \ref{section-contract}.

By symmetry, we may suppose that $p\nmid d_j$ for $1\le j\le r$, and $p\mid d_j$ for $r<j\le \ups_1$; here $0\le r\le \ups_1$ is chosen
appropriately. If $r=1$ or $2$, we take $y_1=y_2=0$ and want to  solve $c+c_3y_3^k+\dots+c_{\ups_1}y_{\ups_1}^k\equiv 0\bmod p$ by Lemma \ref{L1}. For this to be applicable we require that $\ups_1-2\ge p-1$ and $p\nmid c$. When $p\mid c$, take $y_3=\dots=y_{\ups_1}=0$. Hence, in view of \rf{P4-7}, there is always a solution of $c+c_3y_3^k+\dots+c_{\ups_1}y_{\ups_1}^k\equiv 0\bmod p$, and hence of \rf{P4-11}, even when $p=3$.

It remains to consider the situation where $r\ge 3$. Then, by Lemma \ref{L3}, we can rearrange indices $1,\dots,r$ such that $p\nmid c_{2i-1}+c_{2i}$ for $1\le i\le (r-1)/2$. Then choose $y_{2i-1}=d_{2i}z_i$, $y_{2i}=-d_{2i-1}z_i$, with $z_i\in\ZZ$ at our disposal. When $r$ is odd, we put $y_r=0$ and when $r$ is even, we put $y_{r-1}=y_r=0$. The system \rf{P4-11} then reduces to the single congruence 
\be{P4-12}
c+\sum_{1\le i\le (r-1)/2}(c_{2i-1}+c_{2i})z_i^k+\sum_{r<j\le \ups_1}c_jy_j^k\equiv 0\bmod p.
\ee
The coefficients in this congruence are all not divisible by $p$. Further we may suppose that $p\nmid c$ because in the case where $p\mid c$, a solution is provided by $z_i=y_j=0$.

By Lemma \ref{L1}, there is a solution of \rf{P4-12} provided that the congruence involves at least $p-1$ variables. However, the number of variables in \rf{P4-12} is $\ups_1-r+\big[\frac{r-1}{2}\big]$, and $r\le \ups_1$. In particular, we see at least $\tfrac{1}{2}(\ups_1-1)$ variables when $r$ is odd, and at least $\frac{\ups_1}{2}-1$ variables when $r$ is even. By \rf{P4-7}, we have 
$$
\tfrac{1}{2}(\ups_1-1)\ge \frac{\ups_1}{2}-1\ge \tfrac{1}{2}p^2-p-1,
$$
and since $p$ is odd, we conclude that \rf{P4-12} contains at least
\be{P4-13}
\tfrac{1}{2}(p^2-1)-p\ge p-1\quad (p\ge 5)
\ee
variables, hence \rf{P4-12} has a solution. This solutions traces back to a solution of \rf{P4-11}, and to a solution of \rf{P4-10} with one at least of $x_1,\dots, x_{\ups_0}$ not divisible by $p$, as required.

When $p=3$, the inequality \rf{P4-13} fails. Nonetheless, we can still apply the above argument whenever \rf{P4-12} contains at least $2=p-1$ variables, and when $\ups_1\ge 5$, this is always the case. When $\ups_1=4$, and $r\le 3$, we still have $\ups_1-r+\big[\frac{r-1}{2}\big]\ge 2$ variables. When $\ups_1=3$, the cases $r=1=2$ have been successfully dismissed in the initial phase of this discussion.

Recalling that we always have $\ups_1\ge 3$ (from \rf{P4-7}), we infer that even when  $p=3$, we find a solution of \rf{P4-12}, and hence of \rf{P4-10} with one at least of $x_1,\dots,x_{\ups_0}$ not divisible by $3$, except when $\ups_1=r$ is $3$ or $4$. Hence it now remains to consider the congruences \rf{P4-10} with $k=6$, $\ups_1=3$ or $4$, $p\nmid d_1d_2\dots d_{\ups_1}$ and $\ups_0\ge 13-\ups_1$. In these exceptional cases, there is again a solution of \rf{P4-10} with some $x_j$ not divisible by $3$. This is a consequence of the following stronger lemma.

\begin{lemma}\label{lemP4-3}Let $a_1,\dots, a_9, b_1,\dots ,b_9$, $c_1,c_2,c_3,d_1,d_2,d_3$ denote integers, and suppose that $3\nmid a_ic_jd_l$ ($1\le i\le 9$, $1\le j\le 3$, $1\le l\le 3$). Then, there are integers $x_i,y_j$ with 
\be{P4-14}
\begin{array}{rll}
a_1x_1^6+\dots+a_9x_9^6+& 3(c_1y_1^6+c_2y_2^6+c_3y_3^6)&\equiv 0\bmod 9,\\
b_1x_1+\dots +b_9x_9+&\thinspace d_1y_1+d_2y_2+d_3y_3&\equiv 0\bmod 3,\\
\end{array}
\ee
and not all of $x_1,\dots, x_9$ divisible by $3$.
\end{lemma}

Note that once this is established, we have proved that the congruences \rf{P4-10} always have a solution with not all $x_i$ divisible by $p$. Hence, the discussion of type B systems will be complete, and when combined with the results of the previous section, this will also complete the proof of the theorem in the case $k=p(p-1)$, $p\ge 3$.

We now prove Lemma \ref{lemP4-3}. First suppose that $c_1\equiv c_2\equiv c_3\bmod 3$. Then, if we also have $d_1\equiv d_2\equiv d_3\bmod 3$, we take $y_1=y_2=-y_3=z$, with the integer $z$ still at our disposal. If $d_1\equiv d_2\equiv d_3\bmod 3$ does not hold,
then we can arrange indices and suppose that $d_1\equiv d_2\equiv -d_3\bmod 3$, and we take $y_1=y_2=y_3=z$. The congruences \rf{P4-14} then reduce to 
\be{P4-15}
\begin{array}{rll}
a_1x_1^6+\dots+a_9x_9^6& &\equiv 0\bmod 9,\\
b_1x_1+\dots +b_9x_9&+d_1z&\equiv 0\bmod 3.\\
\end{array}
\ee

Similarly, when not all of $c_j$ lie in the same residue class $\bmod 3$, we may suppose that $c_1\equiv c_2\equiv -c_3\bmod 3$. If $d_1\equiv d_3\bmod 3$, we take $y_1=y_3=z$, $y_2=0$ and insert in \rf{P4-14}. We again reduce to \rf{P4-15}, this time with $2d_1$ in place of $d_1$.

By symmetry, the same reduction is possible when $d_2\equiv d_3\bmod 3$. This leaves the case where $d_1\equiv d_2\equiv -d_3\bmod 3$. But then we take $y_1=-y_3=z$, $y_2=0$, and argue as before.

Thus it remains to solve \rf{P4-15}. By Lemma \ref{lemCD}, there are integers $x_1,\dots, x_9$, not all divisible by $3$, that solve the sextic congruence in \rf{P4-15}. With $x_1,\dots, x_9$ now chosen, the linear congruence fixes $z$. It is worth noting that we needed \rf{P4-15} with $d_1$ and $2d_1$ in place of $d_1$, both not divisible by $3$. 

This completes the discussion of the case $k=p(p-1)$.

\section{Powers of $2$: introductory comment}\label{sect-power2-intro}

We now turn to our final task and establish the theorem when 
\be{P5-1}
k=2^{\tau},\quad \tau\ge 2,\quad p=2.
\ee
This will require several new ideas. Most importantly, we will have to rework our basic winning strategy, at least when $k=4$. Thus far, we have followed a traditional path in attacking problems of the type considered in this paper. We began with a conditioned system and then showed that the associated congruences \rf{R12}  possess a solution suitable for an application of Lemma \ref{lemmaH} and Lemma \ref{lemP4-1}. However, when $k=4$,  the strategy necessarily fails for certain systems. Consider     the pair of equations in $18=k^2+2$ variables given by 
\be{P5-2}
x_1^4+\dots+x_{15}^4+8(y_1^4+y_2^4+y_3^4)=y_1+y_2+y_3=0.
\ee
Although this system is certainly not conditioned, one may replace all its zero coefficients by $2^l$, with $l\ge 4$. This yields a family of conditioned systems  of type B, with $\ups_0=15$ and $\ups_3=3$. Whatever the actual value of $l$ may be,  the associated congruences \rf{P4-2} are 
\be{P5-3}
\begin{array}{rll}
x_1^4+\dots+x_{15}^4+&8(y_1^4+y_2^4+y_3^4)&\equiv 0\bmod 16,\\
 &y_1+y_2+y_3&\equiv 0\bmod 2.\\
\end{array}
\ee
Here, the second congruence forces one or three of $y_1,y_2,y_3$ to be even, and in both cases it follows first that $8(y_1^4+y_2^4+y_3^4)\equiv 0\bmod 16$, and then that all $x_j$ must be even. In particular, the pair \rf{P5-3} does not have non-singular solutions. We will therefore have to develop a method that detects such seemingly hopeless examples, and then we  still need to find  $2$-adic solutions in such cases.

Our main tool in this section is a contraction method. The basic ideas go back to Davenport and  Lewis \cite{DL63,DL2}, as developed by Br\"udern and Godinho \cite{BG2}. We require a highly refined version of the methods in  \cite{BG2}, but only in a $2$-adic context.

We now explain in detail  our contraction method, and we also develop a language capable of describing contractions in terms of  a simple formalism. Again, this follows \cite{DL63} in spirit, but considerable refinement will be required.

Let $s\ge 2$, and suppose that a system $A({\bf x})= B({\bf x})=0$ is given by \rf{3-2}. With an application of Lemma \ref{lemP4-1} in mind, we associate with \rf{3-2} the pair of congruences
\be{P5-4}
\sum_ {j=1}^sa_jx_j^k\equiv 0\bmod2^{\tau+2},\quad \sum_{j=1}^sb_jx_j\equiv 0\bmod 2.
\ee

A {\em contraction} of a given system $A=B=0$ is a partition ${\cal C}_1,{\cal C}_2,\dots,{\cal C}_t,{\cal Z}$ of $\{1,\dots,s\}$ with ${\cal C}_j\neq \emptyset$ for $1\le j\le t$. For new variables $y_1,\dots, y_t$, we then take
$$
x_i=y_j\mbox{ for all }i\in {\cal C}_j,\quad x_i=0\mbox{ for all }i\in {\cal Z}
$$
and substitute accordingly in $A({\bf x})=B({\bf x})=0$. We then obtain a new system, say $A'({\bf y})=B'({\bf y})=0$, with
\be{P5-4-1}
A'({\bf y})=\sum_{j=1}^tc_jy_j^k,\qquad B'({\bf y})=\sum_{j=1}^td_jy_j,
\ee
say. We refer to the system $A'=B'=0$ as the system {\em contracted from $A=B=0$ relative to the partition} ${\cal C}_1,{\cal C}_2,\dots,{\cal C}_t,{\cal Z}$. We may take $t=s$, ${\cal Z}=\emptyset$ and ${\cal C}_j=\{j\}$ to see that $A=B=0$ is contracted from itself. Further, if one contracts $A'=B'=0$ to $A''=B''=0$, say,  then the new system $A''=B''=0$ is also contracted from $A=B=0$.

We now focus on {\em preconditioned} systems $A=B=0$ in $s=k^2+2$ variables, with $k=2^{\tau}$ as before. If the system $A'=B'=0$ in variables $y_1,\dots,y_t$ is contracted from $A=B=0$, and the contracted system is given by \rf{P5-4-1}, we refer to the $\nu_j$ defined by $2^{\nu_j}\|c_j$ as the {\em niveau} of the variable $y_j$. We define the {\em parity} of $y_j$ as {\em even} when $2\mid d_j$, and as {\em odd} when $2\nmid d_j$.

Because a preconditioned system given by \rf{3-2} is contracted from itself, niveau and parity of its variables are defined. In particular, its variables of niveau $0$ are precisely those indexed by $i$, where $2\nmid a_i$. For convenience, suppose that this is the set $\{1,\dots, \ups_{0}\}$. If $A'=B'=0$ is contracted from $A=B=0$  relative to ${\cal C}_1,{\cal C}_2,\dots,{\cal C}_t,{\cal Z}$, then we refer to a variable $y_j$ in  \rf{P5-4-1} as {\em primary} when 
${\cal C}_j\cap \{1,\dots,\ups_{0}\}$ is non-empty. Variables that are not primary are {\em secondary}. The relevance of primary variables is illustrated by the following simple observation.

\begin{lemma}\label{lemP5-1}
Let $s\ge k^2+2$, $k=2^{\tau}$ with $\tau\ge 2$, and let $A=B=0$ be a preconditioned system given by \rf{3-2}. If a system $A'=B'=0$ is contracted from $A=B=0$ and contains a primary  even variable at niveau $\tau+2$, then the congruences \rf{P4-2} (with $p=2$) have a solution where  one of the integers $a_ix_i$ with $1\le i\le \ups_0$ is odd.   
If, moreover,  the contraction is relative to ${\cal C}_1,{\cal C}_2,\dots,{\cal C}_t,{\cal Z}$ and the set ${\cal Z}$ contains an index belonging to an odd variable, then the congruences \rf{P4-2} have a non-singular solution.
\end{lemma}

\begin{proof}If $y$ is the contracted primary even variable at niveau $\tau+2$, we take $y=1$ and $y_j=0$ for all other variables in $A'=B'=0$. Tracing this back to $A({\bf x})$, $B({\bf x})$, we obtain a solution of \rf{P4-2} (with $p=2$), with all $x_i\in \{0,1\}$, and $a_lx_l$ odd for at least one $l\in \{1,\dots,\ups_0\}$. When $j\in {\cal Z}$ belongs to an odd variable, the choice $x_j=0$ is forced, and  the  matrix
$$
\left(
\begin{array}{cc}
a_jx_j^{k-1}&a_lx_l^{k-1}\\
b_j&b_l\\
\end{array}
\right)
$$
has determinant $-b_ja_l\equiv 1\bmod 2$. Hence, the solution of \rf{P4-2} is non-singular.
\end{proof}

Later we shall construct the desired primary variable at niveau $\tau+2$ by nested contraction. The following conventions will help to describe the contraction process in an efficient manner. 
A secondary variable in a contracted system at niveau $\nu$ will be denoted $S_{\nu}$. If its parity is known, we write $\Se{\nu}$ when the variable is even, and $\So{\nu}$ when it is odd. 
A primary variable will be denoted as $P_{\nu}$ when its niveau is not lower that $\nu$, and we write $\Pe{\nu}, \Po{\nu}$ when the parity is even, resp. odd. If the variable $P_{\nu}$
is at exact niveau $\nu$, then we signal this by writing $\widehat{P_{\nu}}$.

We are ready to describe the simplest contractions that we shall regularly apply. Given two variables $\Pe{\nu}$, these may be contracted to $\Pe{\nu+1}$. To see this, first consider the case where both variables are $\widehat{\Pe{\nu}}$. If the variables are $x,y$, and they occur in the system with terms $ax^k, bx$, and $a'y^k,b'y$, say, then the contraction $z=x=y$ transfers this to $(a+a')z^k, (b+b')z$. But $b,b'$ are even integers, and so is $b+b'$. Further, $2^{\nu}\|a$, $2^{\nu}\|a'$, and hence $2^{\nu+1}\mid a+a'$, as required. If one of the two $P_{\nu}$ is already a variable of type $P_{\nu+1}$, then we put the other variable to $0$. This confirms the claim. This contraction process we abbreviate as
\be{P5-5}
2\Pe{\nu}\to \Pe{\nu+1}.
\ee
Note that the same argument shows that a $\Pe{\nu}$ and an $\Se{\nu}$ can be contracted to $\Pe{\nu+1}$, and we write this as
\be{P5-6}
\Pe{\nu}, \Se{\nu}\to \Pe{\nu+1}.
\ee
More generally, if ${\cal A}$ is a set of variables in a contracted system, and there is a contraction to a set of variables ${\cal B}$, then we denote this by ${\cal A}\to {\cal B}$. For example, if a conditioned system with $s=k^2+2$ variables is given, then in the notation of Section  \ref{k=p-1/norm}   it contains $\ups_0$ variables $\widehat{P_{0}}$, and $\ups_j$ variables $S_j$ ($1\le j\le k-1$). 
In order to apply Lemma \ref{lemP5-1}, we wish to show that
\be{P5-7}
\ups_0 \widehat{P_{0}},\ups_1 S_1,\dots \ups_{\tau+1} S_{\tau+1}\to \Pe{\tau+2}.
\ee
In  later sections  we shall provide conditions under which \rf{P5-7} is indeed true.

We now turn to the contraction of secondary variables, and begin by showing that
\be{P5-8}
3\Se{\nu}\to \Se{\nu+1},\thinspace \Se{\nu};\quad 3\So{\nu}\to \Se{\nu+1},\thinspace \So{\nu}.
\ee

To see this, let $\pi\in\{{\sf e},{\sf o}\}$, and suppose that $x,y$ are two variables $S_{\nu,\pi}$. These will occur in the associated system with terms $2^{\nu}ax^k,bx$ and $2^{\nu}a'y^k,b'y$. Here $a,a'$ are $\equiv 1\bmod 2$ and $b\equiv b'\bmod 2$. If $a\equiv a'\bmod 4$, we contract the variables via $x=y=z$, and the contraction involves the terms 
$2^{\nu}(a+a')z^k, (b+b')z$. But $2\|a+a'$, so that $z$ is an $\Se{\nu+1}$. If three $S_{\nu,\pi}$ are given, and they occur with $2^{\nu}a_jx_j^k$ in the corresponding system
then the $a_j$ are $\equiv 1\bmod 2$, and we can find two $a_j$ that are in the same residue class modulo $4$. These contract to $\Se{\nu+1}$, leaving one $S_{\nu,\pi}$ unused.

One may repeatedly apply \rf{P5-8} to confirm that for $n\in \NN$ and $\pi\in\{{\sf e},{\sf o}\}$ one has
\be{P5-9}
(2n+1)S_{\nu,\pi}\to n\Se{\nu+1},\thinspace S_{\nu,\pi}.
\ee
Finally, there is a parity-correcting contraction. For $\nu<\mu$, one obviously has 
\be{PC}
\So{\nu},\thinspace \So{\mu}\to\Se{\nu}.
\ee
Similarly, for $j\ge 1$, one has
\be{equation-101}
\hatPe{0},\thinspace \hatPo{0},\thinspace \So{j}\to \Pe{1}.
\ee

\section{Contraction principles}\label{sect-contra-princip}

In this section, we elaborate on the simple examples of contractions presented in the previous section. This will reduce the complexity of the main argument that we present in sections \ref{power-2-type-A} and  \ref{power-2-type-B}  below.

\begin{lemma}\label{lemP5-2}
Let $l\ge 1$, and suppose that for some $\nu\ge 1$, a collection of $2^l$ variables of type $\Pe{\nu}$ and $\Se{\nu}$ is given, with at least one of these primary. Then these variables may be contracted to one  $\Pe{\nu+l}$.
\end{lemma}

\begin{proof}Let $n,m$ be non-negative integers with $n+m=2^l$ and $n\ge 1$. The lemma asserts that
\be{P5-20}
n\Pe{\nu},\thinspace m\Se{\nu}\to \Pe{\nu+l}.
\ee
We prove this by induction on $l$. For $l=1$, the two possible cases $n=1$ and $n=2$ are \rf{P5-5} and \rf{P5-6}.

Now suppose that $l>1$, and that $n+m=2^l$ with $n\ge 1$. If $m=0$, we can apply \rf{P5-5} repeatedly to confirm \rf{P5-20} via $2^l\Pe{\nu}\to 2^{l-1}\Pe{\nu+1}$, and then apply
the induction hypothesis that $2^{l-1}$ of $\Pe{\nu+1}$ will contract to $\Pe{\nu+l}$. If $m\ge 2$ is even, we infer from \rf{P5-9} that $m \Se{\nu}\to \big(\tfrac{m}{2}-1\big)\Se{\nu+1},\thinspace 2\Se{\nu}$, and from $m+n=2^l$ we see that $n$ is even, $n\ge 2$. Hence, we can use \rf{P5-6} twice to conclude
$$
2\Pe{\nu},\thinspace m\Se{\nu}\to 2\Pe{\nu+1},\thinspace \tfrac{1}{2}(m-2)\Se{\nu+1}.
$$

Then, since $n-2$ is even, one may apply \rf{P5-5} repeatedly to see that $(n-2) \Pe{\nu}\to (\frac{n}{2}-1)\Pe{\nu+1}$. When combined with the last display, we have shown that
\be{P5-21}
n\Pe{\nu},\thinspace m\Se{\nu}\to (\frac{n}{2}+1)\Pe{\nu+1},\thinspace (\frac{m}{2}-1)\Se{\nu+1},
\ee
and the desired conclusion \rf{P5-20} follows by applying the case $l-1$ of Lemma \ref{lemP5-2} to the right hand side of \rf{P5-21}.

When $m$ is odd, we first apply \rf{P5-9} and then \rf{P5-6} to confirm that
$$
\Pe{\nu},\thinspace m\Se{\nu}\to \Pe{\nu},\thinspace \Se{\nu},\thinspace \tfrac{1}{2}(m-1)\Se{\nu+1}
\to \Pe{\nu+1},\thinspace \tfrac{1}{2}(m-1)\Se{\nu+1}.
$$
This leaves $n-1$ variables $\Pe{\nu}$ untouched, and since $n$ is odd, repeated use of \rf{P5-5} yields $(n-1)\Pe{\nu}\to \tfrac{1}{2}(n-1)\Pe{\nu+1}$. This shows 
$$
n\Pe{\nu},\thinspace m\Se{\nu}\to \tfrac{1}{2}(n+1)\Pe{\nu+1},\thinspace \tfrac{1}{2}(m-1)\Se{\nu+1}.
$$
Again, appeal to the induction hypothesis completes the proof.
\end{proof}

\begin{lemma}\label{lemP5-3}Let $l\ge 0$, and suppose that $2^{l+1}$ variables of type $\Pe{\nu},\Se{\nu},\dots, \Se{\nu+l}$ are given, with at least $2^l$ of these of type $\Pe{\nu}$. Then a subset of these variables contract to one $\Pe{\nu+l+1}$.

\end{lemma}
\begin{proof}Again, we induct on $l$. The case $l=0$ is covered by Lemma \ref{lemP5-2}. When $l\ge 1$, we consider two cases. First suppose that the list of given variables contain an $\Se{\nu+l}$. In this case, we choose $2^l$ $\Pe{\nu}$ and apply Lemma \ref{lemP5-2}, asserting $2^l\thinspace \Pe{\nu}\to \Pe{\nu+l}$. Then by \rf{P5-6}, the contraction $\Pe{\nu+l},\Se{\nu+l}\to \Pe{\nu+l+1}$ completes the proof in this case.

If there is no $\Se{\nu+l}$ among the variables, we can split the given variables into two disjoint sets of $2^l$ variables each, both containing at least $2^{l-1}$ $\Pe{\nu}$. By induction hypothesis, the variables in each of the two sets contract to a $\Pe{\nu+l}$, so that we have $2\Pe{\nu+l}$. Reference to \rf{P5-5} completes the induction.
\end{proof}

We now develop the contraction principles announced in Lemmas \ref{lemP5-2} and \ref{lemP5-3} further, to include situations where the secondary variables may be odd. We shall be successful only under more restrictive hypotheses.

\begin{lemma}\label{lemP5-4} Let $l\ge 1$, and suppose that for some $\nu\ge 1$, a collection of $2^l+2$ variables of types $\Pe{\nu}$, $\Se{\nu}$ and $\So{\nu}$ is given, at least two of which are primary. Then, a subset of at most  $2^l$ of these variables contract to one $\Pe{\nu+l}$.

\end{lemma}
\begin{proof}The case $l=1$ is covered by \rf{P5-5}. Suppose then that $l\ge 2$, and that $2^l+2=u+n+m$ where $u$ is the number of $\Pe{\nu}$ and where $n$ and $m$ is the number of $\Se{\nu}$, $\So{\nu}$ respectively. We apply \rf{P5-9} whenever $n,m$ are at least $2$, producing  $\big[ \frac{n-1}{2}  \big]+\big[ \frac{m-1}{2}  \big]$ variables $\Se{\nu+1}$, and leaving either one or two of $\Se{\nu}$, $\So{\nu}$ unused, depending on the parities of $n$ and $m$. For those variables $\Se{\nu}$ that remained, we apply \rf{P5-6},  $\Pe{\nu},\thinspace \Se{\nu}\to \Pe{\nu+1}$, and then contract remaining variables $\Pe{\nu}$, if any, in pairs via \rf{P5-5} to $\Pe{\nu+1}$. In this way, we will have at least one $\Pe{\nu+1}$ (because if one $\Se{\nu}$ remained unused, the contraction $\Pe{\nu},\thinspace \Se{\nu}\to \Pe{\nu+1}$ provided one, and otherwise $\Pe{\nu},\thinspace \Pe{\nu}\to \Pe{\nu+1}$ is applied at least once).

We now count how many variables remain unused at niveau $\nu$. If there are two of $\So{\nu}$ remaining, then by \rf{P5-9}, $m$ must be even, and hence $2\mid u+n$, and all variables $\Pe{\nu}$,  $\Se{\nu}$ will have been contracted in pairs to niveau $\nu+1$. If there is only one $\So{\nu}$ remaining, then $m$ was odd, and so is $u+n$. But then, the $\Pe{\nu}$,  $\Se{\nu}$ contract in pairs until one variable remains. Hence, in both cases, two variables will remain at niveau $\nu$, while at niveau $\nu+1$ we have $2^{l-1}$ variables of types $\Pe{\nu+1}$,  $\Se{\nu+1}$, one of which is primary. We may now apply Lemma \ref{lemP5-2} to complete the proof.
 \end{proof}

We now turn to an analogue of Lemma \ref{lemP5-3} in which the secondary variables  may have both parities. To realise this, we require two additional variables, a phenomenon already familiar from a comparison of Lemmas \ref{lemP5-2} and \ref{lemP5-4}. A more restrictive  novelty is that the secondary 
variables are no longer allowed to invade niveau $\nu+l$. In practice, this limits applicability to the range $k\ge 16$.

\begin{lemma}\label{lemP5-5}
Let $l\ge 1$, and suppose that $2^{l+1}+2$ variables  of types $\Pe{\nu}$, $S_{\nu},\dots$, $S_{\nu+l-1}$ are given, with at least $2^l$ of these primary. Then, a subset of these variables contract to one $\Pe{\nu+l+1}$.
\end{lemma}

\begin{proof} The case $l=1$ is the case $l=2$ of Lemma \ref{lemP5-4}, so that we may suppose that $l\ge 2$. 

The strategy is to contract odd variables to even ones at higher level, and then apply Lemma \ref{lemP5-3}. For $\nu\le j\le \nu+l-1$, let $m_j$ denote the number of $\So{j}$ given, and let $m=m_{\nu}+\dots +m_{\nu+l-1}$ be the number of all odd variables. Further, let $n$ be the number of all even variables given, including the primary ones. Then
\be{P5-21bis}
m+n=2^{l+1}+2.
\ee
We also write $n_{\nu}$ for the number of even variables $\Se{\nu}$ and $\Pe{\nu}$.

We begin by contracting $\So{j}$ in pairs to $\Se{j+1}$. For $0\le m_j\le 2$, let $r_j=m_j$, and for $m_j\ge 3$, let $r_j\in \{1,2\}$ be defined by $m_j\equiv r_j\bmod 2$. Then by \rf{P5-9}, the available $\So{j}$ indeed contract in disjoint pairs to $\Se{j+1}$, leaving $r_j$ of $\So{j}$ unused in this process. Note that the new variables are all at niveau between $\nu+1$ and $\nu+l$.

For $r=1$ or $2$, let 
$$
J_{r}=\{\nu<j\le \nu+l-1\colon r_j=r\}.
$$
Consider the situation where $\#J_{2}\ge 2$. Then we choose a pair $j_1,j_2\in J_{2}$ with $j_1<j_2$, and apply \rf{PC}
twice to generate $2\Se{j_1}$ from the so far unused $2\So{j_1}$, $2\So{j_2}$. This process can be repeated until {\em either all} $\So{j}$ with $j\in J_{2}$
have contracted in disjoint pairs to even secondary variables at niveau between $\nu+1$ and $\nu+l-1$, {\em or} this applies to all $j\in J_2$, $j\neq j_0$, for some specific $j_0\in J_2$, and $2\So{j_0}$ remain untouched. Consistent with these operations, we do not apply any contractions when $\#J_2\le 1$.

Now examine the situation when $\#J_1\ge 2$. Should $J_2$ have left $2\So{j_0}$, then choose $j_1\neq j_2\in J_1$, and first apply the contractions
$$
\So{j_i},\thinspace \So{j_0}\to \Se{j'_i}\qquad (i=1,2)
$$
where $j'_i=\min(j_0,j_i)$, and where it is useful to note that $j_i\neq j_0$ $(i=1,2)$ thanks to the construction of $J_1,J_2$. This removes two elements $j_1,j_2$ from $J_1$, and as long as there are two elements $j_3<j_4$ left in $J_1$, we contract these via $\So{j_3},\thinspace \So{j_4}\to \Se{j_3}$. 
This last process we also apply in the case where the variables collected by $J_2$ have contracted completely. These contractions {\em either} contract all remaining odd variables, {\em or } there is exactly one  $\So{j}$, with some $j\in J_1$, that remains uncontracted.

If, however,  $J_1=\{j_1\}$, then in the case where the process applied to $J_2$ left $2\So{j_0}$ unused, we apply \rf{PC} to yield an $\Se{j}$, leaving one $\So{j_0}$ untouched. If $\#J_1=0$, no further contractions are applied.

We have now completed our contractions from odd to even secondary variables. The variables have been contracted in disjoint pairs, and the new even variables are all at niveau between $\nu+1$ and $\nu+l$. Furthermore, at most two variables $\So{\nu}$ and at most two variables $\So{j}$  for exactly one value of $j\in\{\nu+1,\dots, \nu+l-1\}$ have not been involved in a contraction. 

Let $\kappa$ be the number of these exceptions, so that $0\le \kappa\le 4$. Further, since all contractions are in pairs, we have $m\equiv \kappa\bmod 2$, 
and the number of even variables that we have generated at niveaux $\nu+1,\dots, \nu+l$ equals $(m-\kappa)/2$. In addition, there are already $n-n_{\nu}$
even original seed variables at these niveaux.

In the special case $\kappa=4$, we contract the remaining odd variables in two pairs via $\So{\nu},\thinspace \So{j}\to \Se{\nu}$ (recall that $j>\nu$) to $2\Se{\nu}$, adding  two variables to those counted by $n_{\nu}$. We therefore put $n_{\nu}(\kappa)=n_{\nu}$ for $\kappa\le 3$ but $n_{\nu}(4)=n_{\nu}+2$. 

We now contract the even variables at niveau $\nu$, of which there are now $n_{\nu}(\kappa)$, including at least $2^l$ primary ones. Here we begin by 
\rf{P5-9} and contract available $\Se{\nu}$ in pairs to $\Se{\nu+1}$ until there are at most two $\Se{\nu}$ left uncontracted. For these, we choose the same number of $\Pe{\nu}$ (which is possible since $l\ge 2$), and apply \rf{P5-6} to generate $\Pe{\nu+1}$. There are then only $\Pe{\nu}$ left, and these can be contracted via \rf{P5-5} until at most one $\Pe{\nu}$ is left aside. It transpires that this generates $[n_{\nu}(\kappa)/2]$ new variables at niveau $\nu+1$,
including at least $2^{l-1}$ primary ones. On collecting together, at niveaux $\nu+1,\dots, \nu+l$, we now have a total of $T$ variables, where 
$$
T=[n_{\nu}(\kappa)/2]+(n-n_{\nu})+(m-\kappa)/2.
$$
We show that for $\kappa\neq 3$ one has $T\ge 2^l$. To see this, note that 
$$
T\ge \tfrac{1}{2}\big(n_{\nu}(\kappa)+m-\kappa\big)-\tfrac{1}{2}+n-n_{\nu},
$$
this lower bound being valid for all values of $\kappa$. When $\kappa\le 3$, we infer that
\be{P5-22}
T\ge \tfrac{1}{2}(n+m)+\tfrac{1}{2}(n-n_{\nu})-\tfrac{1}{2}(\kappa+1)=2^l+\tfrac{1}{2}(1+n-n_{\nu}-\kappa).
\ee

When $\kappa\le 2$, then from $n-n_{\nu}\ge 0$ we see that $1+n-n_{\nu}-\kappa\ge -1$, and hence, $T\ge 2^l-\tfrac{1}{2}$. Since $T$ is an integer,
we conclude that $T\ge 2^l$, as we claimed. When $\kappa=4$, use $n_{\nu}(4)=n_{\nu}+2$, and proceed as before to again conclude that $T\ge 2^l$.

This leaves the case $\kappa=3$. Then, \rf{P5-22} yields $T\ge 2^l-1+\tfrac{1}{2}(n-n_{\nu})$, and hence, whenever $n>n_{\nu}$,we also conclude that $T\ge 2^l$. In the exceptional situation where $n=n_{\nu}$, we deduce from \rf{P5-21bis} that $n_{\nu}\equiv m\equiv \kappa \bmod 2$, so that $n_{\nu}$
is odd. In this case, the contraction of the variables at niveau $\nu$ will leave one $\Pe{\nu}$ untouched, and since $\kappa=3$, there is one $\So{\nu}$
and one $\So{j}$ (for some $j>\nu$) remaining as well. Hence, the contraction
$\Pe{\nu},\thinspace \So{\nu},\thinspace \So{j}\to \Pe{\nu+1}$
yields an extra variable at niveau $\nu+1$. But $T=2^l-1$ in the current situation, and we again have $2^l$ variables in total.

We have now proved that the seed variables contract to $2^l$ even variables at niveaux scattered  through $\nu+2,\dots,\nu+l$, including $2^{l-1}$ primary variables. Also, $l\ge 2$ implies that we may apply Lemma \ref{lemP5-3} with $\nu+1$ in place of $\nu$, and $l-1$ in place of $l$. This yields one $\Pe{\nu+l+1}$, as required.
\end{proof}

\section{Powers of $2$ : systems of type A}\label{power-2-type-A}

The sole purpose of this section is to establish the following result.

\begin{lemma}\label{lemP5-6}
Let $k=2^{\tau}$ with $k\ge 4$, and let $s=k^2+2$. Let $A=B=0$ be a conditioned system of type A, given by \rf{3-2}. Then, the associated congruences \rf{P5-4} have a non-singular solution.
\end{lemma}

Once this is established, it follows via Lemma \ref{lemmaH} that systems satisfying the hypotheses of Lemma \ref{lemP5-6} have non-trivial $2$-adic solutions.

\medskip

We approach the claim in Lemma \ref{lemP5-6} through the second clause in Lemma \ref{lemP5-1}. Because the system is of type A, one of the variables $x_1,x_2,\dots, x_{\ups_0}$ must be odd, and by symmetry we may suppose that $x_1$ is odd. Put $x_1=0$. According to Lemma \ref{lemP5-1}, it now suffices to show that the variables indexed by $2\le j\le \ups_0+\ups_1+\dots+\ups_{\tau+1}$ contract to one $\Pe{\tau+2}$. Thus, since for systems of type A all secondary variables are even, we have to confirm that
\be{P5-30}
(\ups_0-1)\widehat{P}_0,\thinspace \ups_1\Se{1},\thinspace \ups_2\Se{2},\thinspace \dots, \ups_{\tau+1}\Se{\tau+1}\to \Pe{\tau+2}.
\ee

We begin by contracting the available $\widehat{P}_0$. For $\pi\in \{{\sf o},{\sf e}\}$, one has $2\widehat{P}_{0,\pi}\to \Pe{1}$. Hence we can form disjoint groups of two $\widehat{P}_0$ of the same parity until no further such pairing is possible. When $\ups_0$ is even, this will leave exactly one $\widehat{P}_0$
unused, and produce $\tfrac{1}{2}\ups_0-1$ variables $\Pe{\nu}$. When $\ups_0$ is odd, we may end up with two variables uncontracted, but at least $\tfrac{1}{2}(\ups_0-3)$ variables $\Pe{1}$ will be generated. Thus we always have at least
\be{P5-31}
\big[\tfrac{1}{2}(\ups_0-2)\big]\thinspace \Pe{1}.
\ee

Further contractions will be applied relative to the size of $\ups_0$. We consider cases.

\medskip
(i) {\em Suppose that } $\ups_0\ge 4k+2$. By \rf{P5-31}, we have $2k$ of  $\Pe{1}$ at our disposal. By Lemma \ref{lemP5-3} with $l=\tau$, we see that 
$2k \Pe{1}\to \Pe{\tau+2}$, completing the proof of \rf{P5-30} in this case.

\medskip
(ii) {\em  Suppose that  $2k+2\le \ups_0\le 4k+1$ and} $k\ge 8$. Then \rf{P5-31} provides $k$ variables $\Pe{1}$. Further, by \rf{3ups}, one has
$$
\ups_0+\ups_1+\ups_2+\ups_3+\ups_4\ge 5k+1,
$$
and hence that $\ups_1+\ups_2+\ups_3+\ups_4\ge k$. We may now apply Lemma \ref{lemP5-3} with $l=\tau$ and $\nu=1$ to contract the available $\Pe{1},
\Se{1},\Se{2},\Se{3}$ and $\Se{4}$ to one $\Pe{\tau+2}$, as required.

\medskip
(iii) {\em Suppose that   $k+2\le \ups_0\le 2k+1$ and} $k\ge 8$. Then \rf{P5-31} yields at least $k/2$ variables $\Pe{1}$. 
Since we have $8\mid k$, repeated use of \rf{P5-5} shows that
\be{P5-32}
\tfrac{1}{2}k\thinspace \Pe{1}\to \tfrac{1}{4}k\thinspace \Pe{2}\to \tfrac{1}{8}k\thinspace \Pe{3}.
\ee

If $\ups_3\ge \tfrac{3}{8}k$, then we can form $\tfrac{1}{8}k$ disjoint groups containing one $\Pe{3}$ and three $\Se{3}$. By Lemma \ref{lemP5-2}, each of these groups contracts to a $\Pe{5}$, so that in total we have $\tfrac{1}{8}k$ variables $\Pe{5}$. By Lemma \ref{lemP5-3}, these contract to one $\Pe{\tau+2}$, as required.

Hence, we may suppose that $\ups_3<\tfrac{3}{8}k$.  However, by \rf{3ups}, we have $\ups_0+\ups_1+\ups_2+\ups_3\ge 4k+1$, and in the current situation, this shows that $\ups_1+\ups_2>\tfrac{13}{8}k$. In this case, we only use the first step in \rf{P5-32}, producing $\tfrac{1}{4}k$ variables $\Pe{2}$. Then, by \rf{P5-9}, we contract the available $\Se{1}$ in pairs to $\Se{2}$ until at most two $\Se{1}$ remain unused. This yields 
$\big[\frac{\ups_1-1}{2}\big]$ variables $\Se{2}$. At niveau $2$, we now have $\tfrac{1}{4}k$ primary variables, and $\big[\frac{\ups_1-1}{2}\big]+\ups_2$
secondary ones. However, $\big[\frac{\ups_1-1}{2}\big]\ge \tfrac{1}{2}\ups_1-1$ so that 
$$
\big[\frac{\ups_1-1}{2}\big]+\ups_2\ge \tfrac{1}{2}\ups_1+\ups_2-1>\tfrac{13}{16}k-1.
$$
Since the left hand side is an integer, it follows that $\big[\frac{\ups_1-1}{2}\big]+\ups_2\ge \tfrac{3}{4}k$, and hence, the variables at niveau $2$ can be
grouped into $\tfrac{1}{4}k$ blocks with one $\Pe{2}$ and three $\Se{2}$, contracting to one $\Pe{4}$ each (by Lemma \ref{lemP5-2} again). This yields a total of $\tfrac{1}{4}k$ $\Pe{4}$, contracting to one $\Pe{\tau+2}$ (by Lemma \ref{lemP5-3}).

\medskip
(iv) {\em Suppose that }  $\ups_0=k+1$. By \rf{P5-31}, we construct $\tfrac{1}{2}k-1$ variables $\Pe{1}$. Note that for $k=4$, just one $\Pe{1}$ is provided. By  \rf{3ups}, we have $\ups_0+\ups_1\ge 2k+1$, whence $\ups_1\ge k$.  We begin by contracting the available $\Se{1}$ in pairs to $\Se{2}$ until exactly $\tfrac{1}{2}k$ of $\Se{1}$ are still uncontracted (when $2\mid \ups_1$) or exactly $\tfrac{1}{2}k-1$ of $\Se{1}$ are uncontracted (when $2\nmid \ups_1$). This generates $\big[\tfrac{1}{2}(\ups_1+1-\tfrac{k}{2}) \big]$ variables $\Se{2}$. We now use the uncontracted $\Se{1}$ and apply \rf{P5-6} to produce $\tfrac{1}{2}k-1$ variables $\Pe{2}$. At niveau $2$, we then have
$$
\big(\tfrac{k}{2}-1\big)\thinspace\Pe{2},\quad \big(\big[\tfrac{1}{2}(\ups_1+1-\tfrac{k}{2}) \big]+\ups_2\big)\thinspace\Se{2}.
$$
However, by \rf{3ups}, one has $\ups_0+\ups_1+\ups_2\ge 3k+1$, whence $\ups_1+\ups_2\ge 2k$, and 
$$
\big[\tfrac{1}{2}(\ups_1+1-\tfrac{k}{2}) \big]+\ups_2\ge \frac{\ups_1}{2}+\ups_2-\frac{k}{4}\ge \tfrac{3}{4}k.
$$
Further, for $k\ge 4$, one also has $\tfrac{1}{2}k-1\ge \tfrac{1}{4}k$, and consequently, at niveau $2$, one can form $ \tfrac{1}{4}k$ disjoint groups with  one $\Pe{2}$ and three $\Se{2}$. The argument given at the end of (iii) shows that this is enough to contract to one $\Pe{\tau+2}$.

\medskip
(v) {\em Suppose that $k=4$ and} $6\le \ups_0\le 17$. We start from \rf{P5-31}. Then at niveau $1$, we have
$$
\big[\frac{\ups_0}{2}\big]-1\mbox{ variables }\Pe{1},\quad \ups_1\mbox{ variables }\Se{1}.
$$
Note that $\ups_0\ge 6$ implies that at least two $\Pe{1}$ are in play. Hence, we may begin by contracting available $\Se{1}$ in pairs to $\Se{2}$; leaving up to two $\Se{1}$ unused. These find a partner $\Pe{1}$ to contract to a $\Pe{2}$. After these contractions, the remaining $\Pe{1}$ contract in disjoint pairs to $\Pe{2}$. In total, this generates 
$\big[\tfrac{1}{2}\big([\tfrac{1}{2}\ups_0]-1+\ups_1\big)\big]$ new variables at niveau $2$, including at least one primary variable. Hence, at niveau $2$, the number of variables is 
\begin{align}
\big[\tfrac{1}{2}\big([\tfrac{1}{2}\ups_0]-1+\ups_1\big)\big]+\ups_2
 &\ge \tfrac{1}{2}\big([\tfrac{1}{2}\ups_0]-1+\ups_1\big)-\tfrac{1}{2}+\ups_2\nonumber\\
&\ge \tfrac{1}{2}\big(\tfrac{1}{2}\ups_0-\tfrac{3}{2}+\ups_1\big)-\tfrac{1}{2}+\ups_2
=\tfrac{1}{4}\ups_0+\tfrac{1}{2}\ups_1+\ups_2-\tfrac{5}{4}.\label{P5-33}
\end{align}
However, by \rf{3ups}, we have $\ups_0+\ups_1+\ups_2\ge 14$, so that there are at least three variables at niveau $2$. If $\ups_3\ge 1$, we may obtain a $\Pe{4}$ by the obvious contractions
$$
\Pe{2},\thinspace \Se{2},\thinspace \Se{3}\to \Pe{3},\thinspace \Se{3}\to \Pe{4}
\quad \text{or} \quad 2\Pe{2},,\thinspace \Se{3}\to \Pe{3},\thinspace \Se{3}\to \Pe{4}
$$ 
while in the complementary case $\ups_3=0$ one has $\ups_0+\ups_1+\ups_2=18$, and \rf{P5-33} delivers four variables at niveau $2$, including a primary one. Now Lemma 
\ref{lemP5-2} again yields a $\Pe{4}$.

\medskip
Recall that by \rf{3ups}, a conditioned system with $s=k^2+2$ has $\ups_0\ge k+1$. Hence, the cases (i-v) exhaust all possible cases covered by Lemma \ref{lemP5-6},
and in all cases we have confirmed \rf{P5-30}. This completes the proof.\qed

\section{Powers of $2$ : Systems of type B}\label{power-2-type-B}

The natural analogue of Lemma \ref{lemP5-6} for systems of type B will not hold true, at least when $k=4$. This we have illustrated with the example in section \ref{sect-power2-intro}.
Nonetheless, we shall follow the pattern of the previous section as far as is possible. For systems of type B, Lemmas \ref{lemP5-4} and \ref{lemP5-5} will have to replace Lemmas \ref{lemP5-2} and \ref{lemP5-3} in our treatment of type A. We require Lemma \ref{lemP5-5} with $l=\tau+2$, and it is then  blind for variables at niveau $\tau+1$. This causes extra difficulties, resulting in a separate treatment of $k=8$ in some cases. Except when $\ups_0$ is very large, the case $k=4$ is so different from what follows that large parts of  its discussion are  postponed to the next section.

Throughout, let $k=2^{\tau}$ with $k\ge 4$ and $s=k^2+2$. We begin with a conditioned system $A=B=0$ of type B, given by \rf{3-2}. By \rf{3ups}, this contains $\ups_0$ 
variables at niveau $0$ where $\ups_0\ge k+1$. We contract these in disjoint pairs with the same parity to $\Pe{1}$. When $\ups_0$ is odd, this yields $(\ups_0-1)/2$ variables $\Pe{1}$. When
$\ups_0$ is even, then either one obtains $\ups_0/2$ of $\Pe{1}$, or only finds $(\ups_0-2)/2$ such contractions but then is left with a pair $\hatPe{0}$, $\hatPo{0}$ of uncontracted variables at niveau $0$. Hence, we find
\be{P5-34}
[\ups_0/2]\mbox{ variables }\Pe{1}
\ee
or 
\be{P5-35}
2\mid \ups_0,\quad \tfrac{1}{2}\ups_0-1\mbox{ variables }\Pe{1},\mbox{ and }\hatPe{0},\thinspace \hatPo{0}.
\ee

\begin{lemma}\label{lemP5-7}
Let $k=2^{\tau}$ with $k\ge 4$ and $s\ge k^2+2$. Let $A=B=0$ be a conditioned system of type B, given by \rf{3-2}. Suppose that $\ups_0\ge 4k$. Then, its variables contract to one $\Pe{\tau+2}$.
\end{lemma}
\begin{proof}
If $\ups_0\ge 4k+1$, we apply \rf{P5-34} and \rf{P5-35} to generate $2k$ $\Pe{1}$. By Lemma \ref{lemP5-2}, these variables contract to one $\Pe{\tau+2}$.

This leaves the case $\ups_0=4k$. Again, if we are able to generate $2k$ $\Pe{1}$, these contract to one  $\Pe{\tau+2}$, as before. By  \rf{P5-35}, we are now reduced to
the case where the variables at niveau $0$ contract to $2k-1$ of $\Pe{1}$, leaving a pair $\hatPe{0}$, $\hatPo{0}$. Since the system is of type B, there is a variable $\So{j}$ at some niveau $j\ge 1$, and then the obvious contraction \rf{equation-101}   produces another $\Pe{1}$. Hence, again we have $2k$ $\Pe{1}$ and the proof is completed as before.
\end{proof}

\begin{lemma}\label{lemP5-8}
Let $k=2^{\tau}$ with $k\ge 8$ and $s\ge k^2+2$. Let $A=B=0$ be a conditioned system of type B, given by \rf{3-2}.  Then, its variables contract to one $\Pe{\tau+2}$.
\end{lemma}

\begin{proof}
For $\ups_0\ge 4k$, this conclusion is part of Lemma \ref{lemP5-7}. For $\ups_0<4k$, we mimic the arguments used in the proof of Lemma \ref{lemP5-6}, and proceed by considering cases.

\medskip
(i) {\em Suppose that $2k<\ups_0<4k$, and that} $k\ge 16$. Here, \rf{P5-34} and \rf{P5-35} guarantee at least $k$ variables $\Pe{1}$. By \rf{3ups}, we have $\ups_0+\dots+\ups_4\ge 5k+1$, whence 
$$
\ups_1+\ups_2+\ups_3+\ups_4\ge k+2.
$$
Consequently, for $k=2^{\tau}$ with $\tau\ge 4$, Lemma \ref{lemP5-5} with $\nu=1$, $l=\tau$  is applicable and yields a $\Pe{\tau+2}$.

\medskip
(ii) {\em Suppose that $2k<\ups_0<4k$, and that} $k=8$. A highly refined version of the preceding argument still applies, as we shall now show. We have already pointed out  that for $k=4$, values of $\ups_0$ slightly less than $4k=16$ cannot be approached by an argument of the type suggested by (i), and when $k=8$, these difficulties reflect  in the fine details that require attention below. We have $17\le \ups_0\le 31$, and the variables at niveau $0$ contract to $\big[\tfrac{1}{2}(\ups_0-1)\big]$ variables $\Pe{1}$. Note that these are at least $8$. Hence, if we were able to show that 
\be{P5-36}
\big[\tfrac{1}{2}(\ups_0-1)\big]+\ups_1+\ups_2+\ups_3\ge 18,
\ee
then it would follow from Lemma \ref{lemP5-5} that the variables at niveau $1$, $2$ and  $3$ contract to one $\Pe{5}$, as is required to complete the proof.
If it were the case that
\be{P5-37}
\ups_0+\ups_1+\ups_2+\ups_3\ge 34,
\ee
then
$$
 \big[\tfrac{1}{2}(\ups_0-1)\big]  +\ups_1+\ups_2+\ups_3\ge 34-\ups_0+\big[\tfrac{1}{2}(\ups_0-1)\big],
$$
and $\ups_0\le 31$ implies \rf{P5-36}. However, by \rf{3ups}, we have
$\ups_0+\ups_1+\ups_2+\ups_3\ge 33$, and so, it remains to consider the case where
\be{P5-38}
\ups_0+\ups_1+\ups_2+\ups_3=33.
\ee
But then
\be{P5-39}
\big[\tfrac{1}{2}(\ups_0-1)\big]+\ups_1+\ups_2+\ups_3=33-\ups_0+\big[\tfrac{1}{2}(\ups_0-1)\big],
\ee
and for $\ups_0\le 29$, again \rf{P5-36} follows.

This leaves the case where $\ups_0=30$ or $31$, and where \rf{P5-38} holds. Then $\ups_1+\ups_2+\ups_3=2$ or $3$, and by \rf{3ups}, we also have 
$\ups_0+\ups_1+\ups_2+\ups_3+\ups_4\ge 5k+2=42$, so that $\ups_4\ge 9$. If there is an even variable among the $S_4$, we may apply crude contractions
to conclude that eight of  the $\big[\tfrac{1}{2}(\ups_0-1)\big]$ variables  $\Pe{1}$ contract to  one $\Pe{4}$ by Lemma \ref{lemP5-2}, and $\Pe{4}$, $\Se{4}\to \Pe{5}$ completes the argument in this case. Hence we may suppose that there at least $9$ $\So{4}$, and since there  are 2 or 3  secondary variables in total at niveaux  $1$, $2$ and $3$, we may correct the parity of these variables  via $\So{j}$, $\So{4}\to \Se{j}$, valid for $j\le 3$ by \rf{PC}. After parity correction, we have $\big[\tfrac{1}{2}(\ups_0-1)\big]$ of $\Pe{1}$, and $2$ or $3$ even secondary variables at niveaux not exceeding 3. The total number of all these variables is still given by \rf{P5-39}, and 
is therefore $17$, but now all variables are even, and Lemma \ref{lemP5-3} delivers a $\Pe{5}$, completing the argument in the case under consideration.

\medskip
(iii) {\em Suppose that} $k<\ups_0\le 2k$. By \rf{P5-34} and \rf{P5-35}, we obtain at least $k/2$ variables $\Pe{1}$ from the variables at niveau  $0$. Then, by \rf{P5-5}, one has the chain of contractions
\be{P5-36bis}
\tfrac{1}{2}k\thinspace \Pe{1}\to \tfrac{1}{4}k\thinspace \Pe{2}\to \tfrac{1}{8}k\thinspace \Pe{3}
\ee
and we recall that $8\mid k$.

We now consider the secondary variables. If  $\ups_1\ge \tfrac{3}{2}k+2$,  then we apply Lemma \ref{lemP5-4}       in the form
$$
\tfrac{1}{2}k\thinspace \Pe{1},\thinspace \big(\tfrac{3}{2}k+2\big) \thinspace S_1\to \Pe{\tau+2},
$$
 completing  the proof of the lemma in this case. Thus, from now on, we may suppose that
\be{P5-37bis}
\ups_1\le \tfrac{3}{2}k+1.
\ee
Here we  contract the available $S_1$ in pairs to $S_2$, disregarding parity of the resulting $S_2$. By \rf{P5-9}, we obtain $\big[\tfrac{1}{2}(\ups_1-1)\big]$ new $S_2$. Hence, in total, at niveau $2$ there are
\be{P5-38bis}
\ups_2+\big[\tfrac{1}{2}(\ups_1-1)\big]\ge \tfrac{1}{2}\ups_1+\ups_2-1
\ee
secondary variables  now available. Note that $8\mid k$ implies $\tfrac{1}{4}k\ge 2$, and hence, whenever $\ups_2+\big[\tfrac{1}{2}(\ups_1-1)\big]\ge \tfrac{3}{4}k+2$, we can apply Lemma \ref{lemP5-4}   in the form
$$
\tfrac{1}{4}k\thinspace \Pe{2},\thinspace \big(\tfrac{3}{4}k+2\big) \thinspace S_2\to \Pe{\tau+2},
$$
to finish the proof in this case. Consequently, we may now suppose that
$$
\ups_2+\big[\tfrac{1}{2}(\ups_1-1)\big]\le \tfrac{3}{4}k+1,
$$
and by \rf{P5-38bis}  this implies that
\be{P5-39bis}
\tfrac{1}{2}\ups_1+\ups_2\le \tfrac{3}{4}k+2.
\ee

We now  involve the variables at niveau $3$. In its simplest form, the argument to follow   will only work for $k\ge 16$, as we now temporarily assume. Begin by contracting the $S_2$ in pairs to $S_3$, disregarding parity. By \rf{P5-38bis}  and \rf{P5-9}, the yields a total of 
\begin{align}
\ups_3+\big[\tfrac{1}{2}\big(\ups_2-1+\big[\tfrac{1}{2}(\ups_1-1)\big]\big)\big]
&\ge \ups_3+\tfrac{1}{2}\big(\ups_2-1+\big[\tfrac{1}{2}(\ups_1-1)\big]\big)-\tfrac{1}{2}\nonumber\\
&\ge \ups_3+\tfrac{1}{2}\ups_2+\tfrac{1}{4}\ups_1-1-\tfrac{1}{2}\label{P5-39A}
\end{align}
variables $S_3$. Once again, provided that there are at least $\tfrac{3}{8}k+2$ of $S_3$ available, we can use
\rf{P5-36bis} and     Lemma \ref{lemP5-4}      in the form
$$
\tfrac{1}{8}k\thinspace \Pe{3},\thinspace \big(\tfrac{3}{8}k+2\big) \thinspace S_3\to \Pe{\tau+2},
$$
to complete the proof of the lemma in the current case. Note that at this point we need that $\tfrac{1}{8}k\ge 2$, which requires us to restrict to $k\ge 16$. But then, we are reduced to the case where 
$$
\ups_3+\big[\tfrac{1}{2}\big(\ups_2-1+\big[\tfrac{1}{2}(\ups_1-1)\big]\big)\big]\le \tfrac{3}{8}k+1
$$
which in turn implies
\be{P5-40}
\ups_3+\tfrac{1}{2}\ups_2+\tfrac{1}{4}\ups_1\le \tfrac{3}{8}k+\tfrac{5}{2}.
\ee
Further, on multiplying \rf{P5-37bis} and \rf{P5-39bis} with $\tfrac{1}{2}$, and adding the results with \rf{P5-40}, we infer that
$$
\ups_1+\ups_2+\ups_3\le \tfrac{3}{8}k+\tfrac{5}{2}+\tfrac{1}{2}\big(\tfrac{3}{4}k+2\big)+\tfrac{1}{2}\big(\tfrac{3}{2}k+1\big)=\tfrac{3}{2}k+4.
$$
However, by \rf{3ups}, we have $\ups_0+\ups_1+\ups_2+\ups_3\ge 4k+1$, and hence that
\be{P5-41}
\ups_1+\ups_2+\ups_3\ge 2k+1
\ee
which is a contradiction when $k\ge 16$. This shows  that we have exhausted all possible cases when $k\ge 16$.

This leaves the case $k=8$ for further discussion. In view of \rf{P5-37bis} and \rf{P5-39bis}, we may restrict attention to the case where
$$
\ups_1\le 13,\quad \tfrac{1}{2}\ups_1+\ups_2\le 8.
$$
From \rf{3ups}, we see that $\ups_0+\ups_1+\ups_2\ge 25$, whence $\ups_1+\ups_2\ge 9$. We now argue as in \rf{P5-38bis}, and contract the available $S_1$ in pairs to $S_2$, disregarding parity. Let $u_2$ 
be the exact number of $S_2$ available after this process, including the $S_2$ counted by $\ups_2$. Then 
$u_2\ge \frac{\ups_1}{2}+\ups_2-1$, and hence $u_2\ge 4$.

First consider the case where 
among the $S_2$ there are at least three with the same parity. A pair of these contracts to an $\Se{3}$. Following this contraction, we contract the remaining $u_2-2$ variables $S_2$ in disjoint pairs to $S_3$, without regarding parity. Then, as in \rf{P5-39A}, at niveau $3$ we now have at least $u_3$ variables $S_3$, where
$$
u_3\ge\ups_3+\tfrac{1}{2}\ups_2+\tfrac{1}{4}\ups_1-\tfrac{3}{2}
$$
including at least one $\Se{3}$. By \rf{P5-41}, 
$$
u_3\ge 17-\tfrac{1}{2}\ups_2-\tfrac{3}{4}\ups_1-\tfrac{3}{2}\ge 13-\tfrac{1}{2}\ups_1-\tfrac{3}{2}\ge 5.
$$
Hence, from \rf{P5-36bis}, we see $\Pe{3}$, $\Se{3}$, $4S_3$ at niveau $3$. If the five $S_3$ here include at least $3\Se{3}$, then 
Lemma \ref{lemP5-2} produces the desired $\Pe{5}$. In the contrary case, we have at least $3\So{3}$, and
we can select two of them to contract to an $\Se{4}$.  The desired $\Pe{5}$ is then provided by
$$
\Pe{3},\thinspace \Se{3},\thinspace \Se{4}\to \Pe{4},\thinspace \Se{4}\to \Pe{5}.
$$

If we do not have three $S_2$ with the same parity, then the condition that $u_2\ge 4$ implies that
$u_2=4$, with $2\Se{2}$, $2\So{2}$. In this case, we apply \rf{P5-36bis} and start with $2\Pe{2}$, followed by 
$2\Pe{2}$, $2\Se{2}\to \Pe{4}$ (Lemma \ref{lemP5-2}). However, $u_2=4$ implies $\tfrac{1}{2}\ups_1+\ups_2\le 5$. But then, by \rf{P5-41},
$$
\ups_3\ge 17-\ups_1-\ups_2\ge 7,
$$
so that we can find a pair of $S_3$ of the same parity contracting to an $\Se{4}$. The argument is now completed with $\Pe{4},\thinspace \Se{4}\to \Pe{5}$.
\end{proof}
By Lemmas \ref{lemP5-1}, \ref{lemP4-1} and \ref{lemmaH}, it follows that systems of type B considered in Lemmas    \ref{lemP5-7} and   \ref{lemP5-8} have non-trivial solutions in $\QQ_2$. This completes the proof of our theorem  when $k=2^{\tau}$, $k\ge 8$.

\section{Systems of type B when  $k=4$}\label{sec-typeB-4}

In \rf{P5-2}  we presented an example of a conditioned system with $k=4$ and $s=18$ where the associated congruences \rf{P5-3}  do not admit a non-singular solution. Note that in this example there are three odd variables at niveau $3$.

It turns out that this is typical for such examples. Anticipating this observation, we set out to show that in all other relevant cases, we can still follow the pattern of our work in sections \ref{sect-contra-princip}-\ref{power-2-type-B}. Thus, our leading parameter remains $\ups_0$, but we now closely monitor the variables at niveau $3$. Throughout, we now restrict to the case $k=4$, type B.

\begin{lemma}\label{lem16-1}
Let $k=4$ and $s\ge 18$. Let $A=B=0$ be a conditioned system of type B, given by \rf{3-2}. Suppose that the system includes a variable $\Se{3}$. Then its variables contract to one $\Pe{4}$.
\end{lemma}

\begin{proof} It will suffice to contract the variables at niveaux $0$, $1$ and $2$ to one $\Pe{3}$ because then the contraction 
$\Pe{3},\thinspace \Se{3}\to \Pe{4}$ establishes the lemma.

If $\ups_0\ge 9$, then \rf{P5-34} produces $4\Pe{1}$. If $\ups_0=8$ we apply \rf{P5-35} to produce $3\Pe{1}$,$\hatPe{0}$, $\hatPo{0}$. However, since the system is of type B, there is a variable $\So{j}$, for some $1\le j\le 3$. Now
$  \hatPe{0}, \thinspace\hatPo{0},\thinspace \So{j}\to \Pe{1}$, so that again we have $4\Pe{1}$. Hence, by \rf{P5-5}, whenever $\ups_0\ge 8$, we may contract via $4\Pe{1}\to 2\Pe{2}\to \Pe{3}$. 

We are left with the case where $\ups_0\le 7$. However, by \rf{3ups}, we now have
\be{16-1}
 \ups_0\ge 5,\quad \ups_0+\ups_1\ge 9,\quad \ups_0+\ups_1+\ups_2\ge 14.
\ee
Hence, by \rf{P5-34}, we get $2\Pe{1}$. If $\ups_1\ge 4$, the desired $\Pe{3}$ is implied by Lemma \ref{lemP5-4}. Hence, in view of \rf{16-1}, we may now suppose that $2\le \ups_1\le 3$ and $\ups_2\ge 4$. If among the variables at niveau $2$ there is an $\Se{2}$, we may use
$$
2\Pe{1},\thinspace \Se{2}\to \Pe{2},\thinspace \Se{2}\to \Pe{3}.
$$
Hence, we now suppose that there are $\ups_2$ odd variables at niveau $2$. We now apply $\So{1},\thinspace \So{2}\to \Se{1}$ whenever necessary to construct two variables $\Se{1}$ from the variables initially at niveaux $1$ and $2$. Then 
$2\Pe{1},\thinspace 2\Se{1}\to \Pe{3}$ is a consequence of Lemma \ref{lemP5-2}. This completes the proof.
\end{proof}

From now on, we may suppose that the variables at niveau $3$, if any, are all odd. If there are at most two such variables, then we conclude as follows.

\begin{lemma}\label{lem16-2}
Let $k=4$ and $s\ge 18$. Let $A=B=0$ be a conditioned system of type B, given by \rf{3-2}. Suppose that $\ups_3\le 2$. Then its variables contract to one $\Pe{4}$.
\end{lemma}
\begin{proof}
In view of Lemma \ref{lem16-1}, we may suppose that all variables at niveau $3$ are odd. Further, by Lemma \ref{lemP5-7}, it suffices to study the situation where $\ups_0\le 15$. Also, we have the inequalities \rf{16-1} at our disposal. 
We now divide into cases.

\medskip
(i)  $14\le \ups_0\le 15$. We shall see that a preliminary contraction always yields $7\Pe{1}$, and one $\Se{j}$ for some $1\le j\le 3$. Once this is established, Lemma  \ref{lemP5-3} produces the desired $\Pe{4}$.

If $\ups_0=15$, then $7\Pe{1}$ flow from \rf{P5-34}, and $s\ge 18$ yields at least three secondary variables. Since $\ups_3\le 2$, not all of these can be at niveau $3$. Further, if one of these is even, then we have already reached our goal. Hence, the secondary variables can be assumed to be all odd. If there is an $\So{i}$ and an $\So{j}$ with $1\le i<j\le 3$, then $\So{i},\thinspace \So{j}\to \Se{i}$ yields the desired even variable. Otherwise, we must have $3\So{i}$ for some $i=1$ or $2$. But then \rf{P5-8} yields one $\Se{i+1}$, completing the argument in this case.

If $\ups_0=14$, we recall that the system is of type B, so that    \rf{P5-34} or   \rf{P5-35} and  \rf{equation-101}  produce $7\Pe{1}$, leaving three secondary variables unused. 
As in the case $\ups_0=15$, one contracts two of the unused variables to an even secondary variable, and then proceeds as before.

\medskip
(ii) $\ups_0=13$. Here \rf{P5-34} yields $6\Pe{1}$. If $\ups_1+\ups_2\ge 4$, it suffices to apply Lemma \ref{lemP5-5} to create a $\Pe{4}$. However, $\ups_1+\ups_2\le 3$ together with $s\ge 18$ and $\ups_3\le 2$  implies that $\ups_1+\ups_2=3$, $\ups_3=2$. Since the two variables at niveau $3$ are both odd, we may use \rf{PC} to correct the parity of two of the variables counted by $\ups_1+\ups_2$ to become even. But then we have $6\Pe{1}$, and two even secondary variables at niveau not exceeding $2$. By Lemma 
\ref{lemP5-3}, this yields $\Pe{4}$.

\medskip
(iii) $\ups_0=12$. Here, we first use \rf{P5-34} and \rf{P5-35} to generate $5$  (sic!) $\Pe{1}$. If $\ups_1+\ups_2\ge 5$, Lemma \ref{lemP5-5} creates a $\Pe{4}$. Thus, we may suppose that $\ups_1+\ups_2\le 4$, and again, this implies $s=18$, $\ups_3=2$, $\ups_1+\ups_2=4$.

The variables counted by $\ups_3$ are odd, and we use this to construct a sixth $\Pe{1}$ via  \rf{equation-101}. Hence, we now have $6\Pe{1}$ and $\ups_1+\ups_2=4$, so that Lemma \ref{lemP5-5} again yields a $\Pe{4}$.

\medskip
(iv)  $8\le \ups_0\le 11$. From $\ups_3\le 2$ we have $\ups_0+\ups_1+\ups_2\ge 16$. Further, if $\ups_0$ is odd, we  apply \rf{P5-34} to generate $[\ups_0/2]\ge 4$ variables $\Pe{1}$, and we also have
$$
[\ups_0/2]+\ups_1+\ups_2\ge 10.
$$
We may therefore apply Lemma \ref{lemP5-5} to generate a $\Pe{4}$.

If $\ups_0$ is even, then we apply \rf{P5-35} together with  \rf{equation-101}  to generate $\ups_0/2\ge 4$ of $\Pe{1}$. Note that this is possible since the system is of type B. However, the contractions may involve one secondary variable. 
After this process, at niveaux $1$ and $2$ we  see
$$
\tfrac{1}{2}\ups_0+\ups_1+\ups_2-1\ge 10
$$
variables in total. Hence Lemma \ref{lemP5-5} is applicable, yielding a $\Pe{4}$.

\medskip
(v) $\ups_0=7$. By \rf{P5-34} we get $3\Pe{1}$. Hence, if $\ups_1\ge 7$, Lemma \ref{lemP5-4} provides a $\Pe{4}$ via 
$3\Pe{1},\thinspace 7S_1\to \Pe{4}$. Hence, we may suppose that $2\le \ups_1\le 6$. Now $\ups_3\le 2$ implies  $\ups_2\ge 3$.
We  split into subcases, relating to the available  $S_{2}$.

\medskip
($\alpha$)  Suppose that there are $3\Se{2}$. Then $2\Pe{1}\to \Pe{2}$, and Lemma \ref{lemP5-2}  supplies $\Pe{2},\thinspace 3\Se{2}\to \Pe{4}$, as required.

\medskip
($\beta$)  Suppose that there are $3\So{2}$. We contract these to one $\Se{3}$, leaving one $\So{2}$ uncontracted. This variable we use in  $\So{1},\thinspace \So{2}\to \Se{1}$ if necessary to ensure that there is an $\Se{1}$ available. Now Lemma \ref{lemP5-2} and \rf{P5-6} give
\be{equation-100}
3\Pe{1},\thinspace \Se{1},\thinspace \Se{3}\to \Pe{3},\thinspace \Se{3}\to \Pe{4}.
\ee

\medskip
($\gamma$)  Suppose that the system is not covered by ($\alpha$) and ($\beta$). Then, there are at most two variables $\So{2}$, and at most two $\Se{2}$, and so, $3\le \ups_2\le 4$ and $\ups_1\ge 5$. If there are $2\Se{2}$, then there is also a least one $\So{2}$, and as in case ($\beta$), this odd variable can be used to ensure one $\Se{1}$. But then we complete the argument via
\be{16-A}
3\Pe{1},\thinspace \Se{1},\thinspace 2\Se{2}\to 2\Pe{2},\thinspace 2\Se{2}\to \Pe{4}.
\ee

This leaves the case $\ups_2=3$, with $\Se{2}$, $2\So{2}$ for discussion. Now $\ups_1\ge 6$. The more frequent parity of the variables at niveau $1$ occurs at least three times, and two of them contract to a second $\Se{2}$. This leaves four variables at niveau $1$, and by using one of the $\So{2}$ if necessary, we can ensure that we have an $\Se{1}$ available. We can now complete the argument via \rf{16-A}.

\medskip
(vi) $\ups_0=6$. This is similar to case (v), but there are certain details that require attention. We begin with  \rf{P5-34} and \rf{P5-35}, providing
$2\Pe{1}$, $\hatPo{0}$, $\hatPe{0}$ or $3\Pe{1}$. If $\ups_1\ge 8$, then Lemma \ref{lemP5-4} again yields a $\Pe{4}$. If $\ups_1=7$ and there is a variable $\So{j}$ with $j\ge 2$, then  use \rf{equation-101}, so that we have $3\Pe{1}$ available. Again Lemma \ref{lemP5-4} yields a $\Pe{4}$. Otherwise, all variables at niveaux $2$ and $3$ are even, and $\ups_1=7$ implies $\ups_2\ge 3$, providing $3\Se{2}$, and $2\Pe{1}\to \Pe{2}$. In this case $\Pe{2},\thinspace 3\Se{2}\to \Pe{4}$ yields the desired conclusion. Hence, we are reduced to the case where
$$
3\le \ups_1\le 6,\quad \ups_2\ge 4.
$$
We now follow the argument given in case (v).

\medskip
($\alpha$)   Suppose that there are $3\Se{2}$. Here, as above
\be{16-B}
2\Pe{1},\thinspace 3\Se{2}\to  \Pe{2},\thinspace 3\Se{2}\to \Pe{4}.
\ee
completes the argument.

\medskip
($\beta$)  Suppose that there are $3\So{2}$. These contract to $\Se{3}$, $\So{2}$, and the remaining $\So{2}$ can be used in \rf{equation-101}  to ensure that we have $3\Pe{1}$. 

If there is an $\Se{1}$, then  \rf{equation-100} yields a   $ \Pe{4}$.

In the alternative case, we have at least $3\So{1}$, providing an $\Se{2}$. Now
\be{16-C}
2\Pe{1},\thinspace \Se{2},\thinspace \Se{3}\to \Pe{2},\thinspace \Se{2},\thinspace \Se{3}\to  \Pe{4}.
\ee

\medskip
($\gamma$)  If the system is not covered by ($\alpha$) or ($\beta$), we see from $\ups_2\ge 4$ that we must have $\ups_2=4$ with $2\So{2}$, $2\Se{2}$. But now $\ups_1=6$, and as in case (v), one then may construct an $\Se{2}$ from the variables at niveau $1$. 
One $\Pe{4}$ now comes from \rf{16-B}.

\medskip
(vii) $\ups_0=5$. Here \rf{P5-34} yields $2\Pe{1}$. If $\ups_1\ge 8$ then Lemma \ref{lemP5-4} gives a $\Pe{4}$. Hence, we are reduced to the case where
$$
4\le \ups_1\le 7,\quad \ups_2\ge 4.
$$

\medskip
($\alpha$)  If there are $3\Se{2}$, we use \rf{16-B} to get $\Pe{4}$.

\medskip
($\beta$)  If there are $3\So{2}$, transform theses to $\So{2},\Se{3}$. Should there be $2\Se{1}$, then
$$
2\Pe{1},\thinspace 2\Se{1},\thinspace \Se{3}\to  \Pe{3},\thinspace \Se{3}\to \Pe{4}.
$$

In the alternative case, $\ups_1\ge 4$ yields at least $3\So{1}$, and these contract to $\Se{2}$. Now \rf{16-C} completes the argument.

\medskip
($\gamma$)  If the system is not covered by ($\alpha$) or ($\beta$), then $\ups_2=4$, with $2\Se{2}$, $2\So{2}$. We use the $2\So{2}$ to ensure $2\Se{1}$ at niveau $1$, and then 
$$
2\Pe{1},\thinspace 2\Se{1},\thinspace 2\Se{2}\to  2\Pe{2},\thinspace 2\Se{2}\to \Pe{4}.
$$
The proof if Lemma \ref{lem16-2} is now complete.
\end{proof}

It is perhaps of interest to inspect the role of the variables at niveau $3$ in the proof of Lemma \ref{lem16-2}. While these are essential in the case where $\ups_0=15$, in the case $\ups_0\le 12$ it is only required that there are at most two such variables, their parity is irrelevant, and they are not used in the contractions.

Since we treat type B, a variable $\Pe{4}$ gives  a non-singular solution of the congruences \rf{P5-4} by Lemmas \ref{lemP4-1} and \ref{lemP5-1}, and hence, the given system has a $2$-adic non-trivial solution by Lemma  \ref{lemmaH}  in the cases covered by Lemmas \ref{lem16-1} and \ref{lem16-2}. Therefore  it only remains to discuss
conditioned systems with $s\ge 18$, and $\ups_3\ge 3$ where all variables at niveau $3$ are odd.

\section{Cycling home}\label{sec-cycling-home}

We now embark on our final task. In order to complete  the proof of the Theorem when $k=4$, it remains to show that a conditioned
system with $k=4$, $s\ge 18$ and $\ups_3\ge 3$ with all variables at niveau $3$ odd, has  non-trivial $2$-adic solutions.
Note that \rf{P5-3} is such a system, forcing us to waive the strategy followed in section \ref{sec-typeB-4}.

Instead, we apply a ``cycling trick'', inspired by the proof of Lemma \ref{L8}. Suppose that $A=B=0$ is a conditioned system satisfying the conditions described in the previous paragraph. Then, by \rf{3ups}, we have
$$
\ups_0\ge 5,\quad \ups_0+\ups_1\ge 9,\quad \ups_0+\ups_1+\ups_2\ge 14,
$$
and $\ups_3\ge 3$ by hypothesis. Let ${\bf x}_0,\dots,{\bf x}_3$ be as in \rf{3}. The system $A({\bf x})=B({\bf x})=0$ is equivalent
with the system 
\be{17-1}
\tfrac{1}{8}A(2{\bf x}_0,2{\bf x}_1,2{\bf x}_2,{\bf x}_3)=B(2{\bf x}_0,2{\bf x}_1,2{\bf x}_2,{\bf x}_3)=0,
\ee
and observe that $\frac{1}{8}A(2{\bf x}_0,2{\bf x}_1,2{\bf x}_2,{\bf x}_3)$ is a form with integer coefficients.

We put ${\bf y}_j={\bf x}_{j-1}$ ($1\le j\le 3$), and ${\bf y}_0={\bf x}_3$. Then, in the language introduced in section \ref{sect-power2-intro}, the variables ${\bf y}_j$ are now at niveau $j$. Also, all variables ${\bf y}_0$ are odd, thanks to our overall hypothesis. Further, the variables ${\bf y}_1$, ${\bf y}_2$, ${\bf y}_3$ are all even, by construction.

Note that the system \rf{17-1} is \textit{not} conditioned. However, all its coefficients are still non-zero, and we have $\ups_3$ variables $\hatPo{0}$, and $\ups_{j-1}$ variables $\Se{j}$ ($1\le j\le 3$). We now argue as follows. We first use 
$3\hatPo{0}\to \Pe{1},\thinspace \hatPo{0}$. 

If $\ups_0\ge 7$, then $\Pe{1},\thinspace 7\Se{1}\to \Pe{4}$ is provided by Lemma \ref{lemP5-2}. If $\ups_0=5$ or $6$, then $\ups_0+\ups_1\ge 9$ implies $\ups_1\ge 3$. We first contract two of the $\ups_0$ $\Se{1}$ to one $\Se{2}$, leaving an $\Se{1}$ behind, and then
$$
\Pe{1},\thinspace \Se{1},\thinspace 3\Se{2}\to  \Pe{2},\thinspace 3\Se{2}\to \Pe{4}.
$$
Hence, in all cases, the variables in the system \rf{17-1} contract to $\Pe{4}$, leaving a $\hatPo{0}$ untouched.

 As in the proof of Lemma \ref{lemP5-1}, this amounts to choosing $y_1=y_2=1$, $y_3=0$ in ${\bf y}_0$, and an inspection of the proof of Lemma \ref{lemP5-1} shows that
we have found a non-singular solution to the congruences \rf{P5-4} associated with \rf{17-1}.

Consequently, the system \rf{17-1} has non-trivial $2$-adic solutions by Lemma \ref{lemmaH}, and so has the original system $A=B=0$. This  establishes the theorem when $k=4$.

\section*{Acknowledgements}
The authors are grateful to their home institutions for support on the occasion of mutual visits during the period where this paper was conceived. The first author acknowledges with gratitude support by Deutsche Forschungsgemeinschaft and Schweizer Nationalfond.
Further, he thanks T. Wooley and V. Kala for  discussions concerning the material in \S2, and M. Kaesberg for comments on a draft version of this paper. Last but not least we wish to express our sincere gratitude to an anonymous referee who has read the manuscript with utmost care, drew our attention to reference \cite{Ell}  and made numerous suggestions that improved the presentation.

\vspace{2ex}\noindent
 J\"org Br\"udern \\
 Universit\"at G\"ottingen\\
Mathematisches Institut\\
Bunsenstrasse 3--5\\
D 37073 G\"ottingen \\
Germany\\
bruedern@uni-math.gwdg.de\\
[2ex]
Olivier Robert\\
 Universit\'e de Lyon \\ and  Universit\'e de Saint-Etienne\\
Institut Camille Jordan CNRS UMR 5208\\
23, rue du Dr P. Michelon\\
F-42000, Saint-Etienne\\
 France\\
olivier.robert@univ-st-etienne.fr


\begin{thebibliography}{99}

\bibitem{kara}  G.I. Arkhipov and A.A. Karatsuba,  Local representation of zero by a form. (Russian) Izv. Akad. Nauk SSSR Ser. Mat. 45 (1981), no. 5, 948--961, 1198

\bibitem{AK}  J. Ax and  S. Kochen, Diophantine problems over local fields. I. Amer. J. Math. 87 (1965) 605--630.

\bibitem{Brow} D. Brownawell, On $p$-adic zeros of forms. J. Number Th. 18 (1984) 342--349

\bibitem{BG1} J.   Br\"udern and H. Godinho, On Artin's conjecture. I. Systems of diagonal forms. Bull. London Math. Soc. 31 (1999), no. 3, 305--313. 

\bibitem{BG2} J.   Br\"udern and H. Godinho, On Artin's conjecture. II. Pairs of additive forms. Proc. London Math. Soc. (3) 84 (2002), no. 3, 513--538.

\bibitem{CMS} S. Chowla, H. B.  Mann and E. G. Straus,  Some applications of the Cauchy-Davenport theorem. Norske Vid. Selsk. Forh. Trondheim 32 (1959) 74--80.

\bibitem{D59}  H. Davenport,  Cubic forms in thirty-two variables. Philos. Trans. Roy. Soc. London. Ser. A 251 (1959) 193--232.

\bibitem{DL63} H. Davenport and D. J. Lewis, Homogeneous additive equations. Proc. Roy. Soc. Ser. A 274 (1963) 443--460.

\bibitem{DL66} H. Davenport and D. J. Lewis, Cubic equations of additive type. Philos. Trans. Roy. Soc. London Ser. A 261 (1966) 97--136.

 \bibitem{DL69} H. Davenport and D. J. Lewis,  Simultaneous equations of additive type. Philos. Trans. Roy. Soc. London Ser. A 264 (1969) 557--595.

\bibitem{DL2}  H. Davenport and D. J. Lewis, Two additive equations. 1969 Number Theory (Proc. Sympos. Pure Math., Vol. XII, Houston, Tex., 1967) pp. 74--98.  Amer. Math. Soc., Providence, R.I. 

\bibitem{D} V. B. Dem'yanov, On cubic forms in discretely normed fields. (Russian) Doklady Akad. Nauk SSSR (N.S.) 74, (1950), 889--891. 

\bibitem{Do} M. Dodson,  Homogeneous additive congruences.  Philos. Trans. Roy. Soc. London Ser. A 261 (1967) 163--210. 

\bibitem{Dumk} J. H. Dumke, $p$-adic zeros of quintic forms, arXiv:1308.0999

\bibitem{Ell} W. J. Ellison, A `Waring Problem' for homogeneous forms. Proc. Cambridge Philos. Soc. 65 (1969) 663--672.

\bibitem{K}  M. P. Knapp, Pairs of additive forms of odd degrees. Michigan Math. J. 61 (2012), no. 3, 493--505

\bibitem{kranz} C. Kr\"anzlein, Paare additiver Formen vom Grad $2^n$. Dissertation, Universit\"at Stuttgart 2009, dx.doi.org/10.18419/opus-4921.

 \bibitem{L} D. J.  Lewis, Cubic homogeneous polynomials over p-adic number fields. Ann. of Math. (2) 56 (1952), 473--478.
 
 \bibitem{LM} D. J. Lewis and H. L. Montgomery, On zeros of p-adic forms. Michigan Math. J. 30 (1983), no. 1, 83--87.

\bibitem{LPW}  L. Low, J.  Pitman and A. Wolff, Simultaneous diagonal congruences. J. Number Theory 29 (1988), no. 1, 31--59

\bibitem{M}  A. Meyer, ''Mathematische Mittheilungen'', Vierteljahrschrift der Naturforschenden Gesellschaft in Z\"urich  29 (1884), 209--222

\bibitem{N} {} M. B. Nathanson, Additive number theory. Inverse problems and the geometry of sumsets. Graduate Texts in Mathematics, 165. Springer-Verlag, New York, 1996. 

\bibitem{O} J. E. Olson, A combinatorial problem on finite Abelian groups. I. J. Number Theory 1 (1969) 8--10.

\bibitem{Rez} B. Reznick, On the length of binary forms. Quadratic and higher degree forms, 207-232, Dev. Math., 31, Springer, New York, 2013.

\bibitem{T1} G. Terjanian, Un contre-exemple \`a une conjecture d'Artin. (French) C. R. Acad. Sci. Paris S\'er. A-B 262 (1966) A612.

\bibitem{T2} G. Terjanian,  Formes $p$-adiques anisotropes. (French) J. Reine Angew. Math. 313 (1980), 217--220.

\bibitem{Wool} T. D. Wooley, On simultaneous additive equations. I. Proc. London Math. Soc. (3) 63 (1991), 1--34.

\bibitem{W3} T. D. Wooley,  Artin's conjecture and systems of diagonal equations. Forum Math. 27 (2015), no. 4, 2259--2265.

\bibitem{Tsurvey} T. D. Wooley, Diophantine problems in many variables: the role of additive number theory. Topics in number theory (University Park, PA, 1997), 49-83, Math. Appl., 467, Kluwer Acad. Publ., Dordrecht, 1999.





 



\end{thebibliography}
\end{document}